\documentclass[reqno, 11pt]{amsart}

\usepackage[margin=1in]{geometry}
\usepackage{
  amsmath,
  amssymb,
  amsthm,
  hyperref,
  paralist,
  xypic,
  graphicx,
  thmtools,
  thm-restate,
}

\usepackage{tikz-cd}
\usepackage{mathrsfs}

\usepackage[center]{subfigure}
\usepackage{setspace}

\usepackage[all]{xy}

\newtheorem*{thm*}{Theorem}
\newtheorem{thm-defn}[thm]{Theorem/Definition}

\newtheorem{question}{Question}

\theoremstyle{definition}

\theoremstyle{remark}

\numberwithin{equation}{section}
\numberwithin{thm}{section}
\numberwithin{prop}{section}
\numberwithin{lem}{section}
\numberwithin{cor}{section}
\numberwithin{conj}{section}
\numberwithin{defn}{section}

\newcommand{\td}{\widetilde}
\newcommand{\F}{{\mathbb F}}

\renewcommand{\to}{\longrightarrow}

\newcommand{\comment}[1]{}

\newcommand {\from}{{\colon}}

\newcommand{\dual}[1]{#1^{\vee}}

\newcommand{\Spec}{{\text{\rm Spec}\,}}
\newcommand{\Bl}{{\text{\rm Bl}\,}}

\newcommand{\rk}{{\text{rk}\,}}

\newcommand {\Pic}{{\rm Pic}\,}
\DeclareMathOperator\Sym{Sym}
\DeclareMathOperator\ch{char}

\DeclareMathOperator\PGL{PGL}

\usepackage{color}
\definecolor{purple}        {cmyk}{0.05,0.8,0,0}

\definecolor{adr}        {cmyk}{0.99,0.,0.,0.1}

\theoremstyle{plain}
\newtheorem{theorem}{Theorem}[section]
\newtheorem{proposition}[theorem]{Proposition}
\newtheorem{lemma}[theorem]{Lemma}
\newtheorem{corollary}[theorem]{Corollary}
\theoremstyle{definition}
\newtheorem{definition}[theorem]{Definition}
\newtheorem{remark}[theorem]{Remark}
\newtheorem{example}[theorem]{Example}

\newcommand\nc{\newcommand}
\nc\on{\operatorname}
\nc\renc{\renewcommand}

\newcommand\ssec{\subsection}
\newcommand\sssec{\subsubsection}

\newcommand\bp{{\mathbb P}}

\newcommand\bz{{\mathbb Z}}

\newcommand\so{{\mathscr O}}

\newcommand\scf{\mathscr F}

\newcommand\sch{\mathscr H}
\newcommand\sci{\mathscr I}

\newcommand\scl{\mathscr L}
\newcommand\scm{\mathscr M}

\newcommand\sco{\mathscr O}

\newcommand\scu{\mathscr U}
\newcommand\scv{\mathscr V}

\newcommand\scx{\mathscr X}
\newcommand\scy{\mathscr Y}
\newcommand\scz{\mathscr Z}

\newcommand \ra{\rightarrow}
\newcommand{\id}{\mathrm{id}}
\newcommand\im{\text{im }}

\newcommand \spec{\text{Spec }}

\newcommand \mg{{\overline{\mathscr M}_g}}

\newcommand \hilb[1]{\sch_{#1}}
\newcommand \uhilb[1]{\scv_{#1}}
\newcommand \ci[3]{\sch^{ci}_{#1, #2, #3}}

\newcommand \stor{\textit{Tor}}

\makeatletter
\newcommand{\customlabel}[2]{%
   \protected@write \@auxout {}{\string \newlabel {#1}{{#2}{\thepage}{#2}{#1}{}} }%
   \hypertarget{#1}{#2}
}

\DeclareMathOperator\Hilb{Hilb}
\DeclareMathOperator\codim{codim}

\usepackage[colorinlistoftodos, textsize=tiny]{todonotes}
\def\listtodoname{List of Todos}
\def\listoftodos{\@starttoc{tdo}\listtodoname}

\title{Interpolation Problems: Del Pezzo Surfaces}

\author{Aaron Landesman} \address{Dept. of Mathematics, Harvard University,
Cambridge, MA 02138 USA} 
\email{aaronlandesman@gmail.com}
\author{Anand Patel} \address{Dept. of
Mathematics, Boston College, Chestnut Hill, MA 02467 USA}
\email{anand.patel@bc.edu}
\date{\today}

\begin{document}

\maketitle

\begin{abstract}
We consider the problem of interpolating projective
varieties through points and linear spaces.
We show that del Pezzo surfaces satisfy weak interpolation.
\end{abstract}

\section{Introduction}
The question of interpolation is one of the most classical
questions in algebraic geometry. Indeed, it dates all the
way back to the ancients, starting with Euclid's postulate
that through any two points there passes a unique line.   The problem of interpolating a polynomial function $y = f(x)$ of degree $\leq n-1$ passing through $n$ general points in the plane was explicitly solved by Lagrange in the late 1700's. These examples should be considered the first two instances of an interpolation problem in the sense of this paper. 

In simple terms, an interpolation problem involves two pieces of data:
\begin{enumerate}
	\item A class $\sch$ of varieties in projective space (e.g. ``rational normal curves'') often specified by a component of a Hilbert scheme
	\item A collection of (usually linear) incidence conditions (e.g. ``passing through five fixed points and incident to a fixed $2$-plane'').
\end{enumerate}
The problem is then to determine whether there exists a variety $[X] \in \sch$ meeting a general choice of conditions of the specified type. 

More precisely,
suppose $U$ is an integral subscheme of the Hilbert scheme
parameterizing varieties of dimension $k$ in $\bp^n$.
Define $q$ and $r$ so that $\dim U = q \cdot (n- k) + r$.
We say $U$ {\bf satisfies interpolation} if for any collection
$p_1, \ldots, p_q, \Lambda$, where $p_i \in \bp^n$ are points
and $\Lambda \subset \bp^n$ is a plane of dimension
$n-k-r$, there exists some $[Y] \in U$ so that
$Y$ passes through $p_1, \ldots, p_q$ and meets $\Lambda$.
We say $U$ {\bf satisfies weak interpolation} if there exists
some $[Y] \in U$ meeting $q$ general points $p_q, \ldots, p_q$.
If $X$ is a projective variety lying on a unique
irreducible component of the Hilbert scheme, denoted
$\hilb X$, then we say $X$ satisfies interpolation if $\hilb X$ does.
Although this description of interpolation,
given in \autoref{theorem:equivalent-conditions-of-interpolation}\ref{interpolation-naive}
is the most classical one,
there are at least twenty two equivalent descriptions of interpolation
under moderate hypotheses, as we show later in \autoref{theorem:equivalent-conditions-of-interpolation}.

The first nontrivial case of interpolation in higher dimensional projective space is that rational normal curves satisfy interpolation,
meaning there is one through a general collection of $n + 3$ points in $\bp^n$, see subsection \ref{sssec:castelnuovos-lemma}.
Interpolation of higher genus curves in projective space is extensively
studied in 
\cite{stevens:on-the-number-of-points-determining-a-canonical-curve}, \cite{atanasovLY:interpolation-for-normal-bundles-of-general-curves}, \cite{atanasovLY:interpolation-for-normal-bundles-of-general-curves}, and \cite{larson:interpolation-for-restricted-tangent-bundles-of-general-curves}. We review interpolation
for rational curves and results of interpolation for higher genus curves in \autoref{section:lay-of-land} below.

Surprisingly, despite being such a  fundamental problem,
very little is known about interpolation
of higher dimensional varieties in projective space.  To our knowledge, the work of Coble in \cite{coble:associated-sets-of-points}, of
Coskun in \cite{coskun:degenerations-of-surface-scrolls}, and of Eisenbud and Popescu in
\cite[Theorem 4.5]{eisenbudP:the-projective-geometry-of-the-gale-transform} 
are the only places where a higher dimensional interpolation problem is addressed. In \cite[Theorem 1.1]{landesman:interpolation-of-varieties-of-minimal-degree}, the first author showed all varieties of minimal degree satisfy
interpolation.  
In this paper, we continue the study of interpolation problems for higher dimensional
varieties:
\begin{restatable}{theorem}{main}
	\label{theorem:main}
All del Pezzo surfaces satisfy weak interpolation.
\end{restatable}

In many ways, the del Pezzo surfaces are a natural next
class of varieties to look at. First, as mentioned earlier, varieties of
minimal degree were shown to satisfy interpolation
in the first author's \autoref{thm:stronginterpscroll}. Del Pezzo surfaces
are surfaces of degree $d$ in $\bp^d$, one higher
than minimal. Further, all irreducible surfaces of
degree $d$ in $\bp^d$ are either del Pezzo surfaces,
projections of surfaces of minimal degree, or cones
over elliptic curves, by \cite[Theorem 2.5]{coskun:the-enumerative-geometry-of-del-Pezzo-surfaces-via-degenerations}. So, del Pezzo constitute all 
linearly normal smooth varieties of
degree $d$ in $\bp^d$.
By analogy, all curves of degree $d-1$ in $\bp^d$
(also one more than minimal degree) have been shown
to satisfy interpolation in \cite[Theorem 1.3]{atanasovLY:interpolation-for-normal-bundles-of-general-curves}.
Second, since del Pezzos are the only Fano surfaces, they can
be viewed as the surface analogue of rational curves, which are already
known to satisfy interpolation.

\ssec{Relevance of Interpolation}

Before detailing what is currently known about interpolation,
we pause to describe several ways in which interpolation
arises in algebraic geometry.

First, interpolation arises naturally when studying families of varieties.  As an example, we consider the problem of producing {\sl moving curves} in the moduli space of genus $g$ curves, 
$\mg$. 
Suppose we know, for example, that canonical (or multi-canonical) curves
satisfy interpolation through a collection of points
and linear spaces. Then, after imposing the correct number of incidence conditions, one obtains a moving curve in $\mg$.
Indeed, as one varies the incidence conditions, these
curves sweep out a dense open set in $\mg$, and hence
determine a moving curve.
One long-standing open problem in this area  is that of determining the least upper bound for
the slope $\delta/\lambda$ of a moving
curve in $\mg$.
In low genera, moving
curves constructed via interpolation realize the least upper bounds. Establishing interpolation is a necessary first step
in the construction of such moving curves. For a more in depth discussion of slopes, see \cite[Section 3.3]{chenFM:effective-divisors-on-moduli-spaces-of-curves-and-abelian-varieties}.
This application is also outlined in the second
and third
paragraphs of \cite{atanasov:interpolation-and-vector-bundles-on-curves}.

We next provide an application of interpolation to the problems in Gromov-Witten theory.
Gromov-Witten theory can be used to count the number of curves satisfying incidence or tangency conditions.
Techniques in interpolation can also be used to count this number,
and we now explain how interpolation techniques can sometimes
lead to solutions where Gromov-Witten Theory fails.
When the Kontsevich space is irreducible and of the correct dimension 
one can employ Gromov-Witten theory without too much
difficulty to count the number of varieties meeting a certain
number of general points. 
In more complicated cases, one needs a virtual fundamental class,
and then needs to find the contributions of this
virtual fundamental class from nonprincipal components
and subtract the contributions from these components.
However, arguments in interpolation can very often be used
to count the number of varieties containing a general set of points,
as is done for surface scrolls in \cite[Results, p.\ 2]{coskun:degenerations-of-surface-scrolls}.
Coskun's technique also allows one to efficiently compute
Gromov-Witten invariants for curves in $\mathbb G(1,n)$.
Although there was a prior algorithm to compute this using
Gromov-Witten theory, Coskun notes that his method
is exponentially faster. The
standard algorithm, when run on Harvard's MECCAH cluster
``took over four weeks to determine the cubic invariants
of $\mathbb G(1,5)$. The algorithm we prove here allows
us to compute some of these invariants by hand'' \cite[p.\ 2]{coskun:degenerations-of-surface-scrolls}.
 
 Interpolation also distinguishes components of Hilbert schemes. For a typical example of this phenomenon,
consider the Hilbert scheme of twisted cubics in $\bp^3$.
This connected component of the Hilbert scheme has two
irreducible components.
One of these components has general member which is a smooth
rational normal curve in $\bp^3$ and is $12$ dimensional.
The other component
has general member corresponding to the union of a plane
cubic and a point in $\bp^3$, which is $15$ dimensional.
While the component of rational normal curves satisfies interpolation
through $6$ points, the other component doesn't 
even pass through $5$ general points, despite having
a larger dimension than the component parameterizing
smooth rational normal curves.

\ssec{Interpolation: a lay of the land}\label{section:lay-of-land}

\sssec{Rational normal curves}
\label{sssec:castelnuovos-lemma}

Interpolation holds for rational normal curves.
This is precisely the well known fact that through $r+3$ general points in $\bp^{r}$ there exists a unique rational normal curve $\bp^{1} \subset \bp^{r}$.  A dimension count provides evidence for existence: the (main component of the) Hilbert scheme of rational normal curves is $r^{2}+2r-3 = (r+3)(r-1)$ dimensional, and the requirement of passing through a point imposes $r-1$ conditions on rational normal curves. Therefore we {\sl expect} finitely many rational normal curves through $r+3$ general points.  
``Counting constants'' as above  only provides a plausibility argument for existence of rational curves interpolating through the required points -- it is not a proof. To illustrate this, we give an example where interpolation is not satisfied, even though the dimension count says otherwise.

\begin{example}
	\label{example:}
	A parameter count suggests there should be a genus $4$ canonical curve through $12$ general points in $\bp^{3}$: The dimension of the Hilbert scheme of canonical curves is $\dim \mathscr M_4 + \dim Aut(\bp^3) = 3 \cdot 4 - 3 + 4^2 - 1 = 24$ and each point imposes two conditions on a curve in $\bp^3$, so that
	we expect a $0$ dimensional family through $12 = 24/2$ points. However, such a canonical
	curve is a complete intersection of a quadric and a cubic.
	Since
	a quadric is determined by $9$ general points,
	the curve,
	which lies on the quadric, cannot contain $12$ general points.
	In other words, canonical genus $4$ curves do not satisfy interpolation.
\end{example}

There are many proofs that there is a unique rational normal curve through $r + 3$ points in $\bp^r$.  One proof proceeds by directly constructing a rational normal curve using explicit equations.  Another approach is via a degeneration argument, as in \autoref{example:degeneration-rnc}. One can also use {\sl association} (see \cite{eisenbudP:the-projective-geometry-of-the-gale-transform}) to deduce the lemma.  A purely synthetic proof also exists, as is found in \cite[Proposition 2.4.4]{pereiraP:an-invitation-to-web-geometry}.
\subsubsection{Higher genus curves} One way to generalize interpolation
for rational normal curves is to consider higher genus curves in projective space. For many reasons it is simpler to consider curves embedded via nonspecial linear systems.  
Interpolation for arbitrary rational curves, not just rational normal curves, was proven in \cite{sacchiero:normal-bundles-of-rational-curves-in-projective-space},
and later independently proven in \cite{ran:normal-bundles-of-rational-curves-in-projective-spaces}.
Hence, it is natural to ask whether curves of higher genus satisfy interpolation.
The related question of semistability for curves of genus 1 was explored in
\cite{einL:stability-and-restrictions-of-picard-bundles},
which was later used in
\cite[Theorem 1]{ballico2014interpolation} to prove that elliptic normal curves satisfy interpolation.

Around the same time, it was shown in \cite[Theorem 7.1]{atanasov:interpolation-and-vector-bundles-on-curves} that nonspecial curves, apart from those of genus $2$ and degree $5$, in $\mathbb P^3$ satisfy interpolation.
This was generalized from $\mathbb P^3$ to projective spaces of arbitrary dimension in the following comprehensive recent result of Atanasov--Larson--Yang:

\begin{theorem}[Theorem 1.3, \cite{atanasovLY:interpolation-for-normal-bundles-of-general-curves}]\label{thm:NaskoEricDavid}
Strong interpolation holds for the main component of the Hilbert scheme parameterizing nonspecial curves of degree $d$, genus $g$ in projective space $\bp^{r}$, with $d \geq g + r$ unless
\begin{align*}
	(d,g,r) \in \left\{ (5,2,3), (6,2,4), (7,2,5) \right\}.
\end{align*}
\end{theorem}

It is also shown in 
\cite[p.\ 108]{stevens:deformations-of-singularities}
(which combines the work in ~\cite{stevens:on-the-number-of-points-determining-a-canonical-curve},
dealing with the canonical curves of genus not equal to $8$ and
~\cite[Proposition, p.\ 3715]{stevens:on-the-computation-of-versal-deformations},
dealing with canonical curves of genus 8)
that canonical curves of genus at least $3$ fail to satisfy weak interpolation if and only if their genus is $4$ or $6$.

In fact, the above leads to a complete classification of
whether Castelnuovo curves satisfy weak interpolation:
\begin{example}
	\label{example:}
	Castelnuovo curves of degree $d$ and genus $g$ in $\bp^r$
	satisfy weak interpolation if and only if
	$d \leq 2r$ and
	$(d,g,r) \notin \left\{ (5,2,3), (6,2,4), (7,2,5), (6,4,3), (10,6,5) \right\}$.
	Further, a Castelnuovo curve of degree $d$ and genus $g$
	in $\bp^r$ of degree not equal
	to $2r$ satisfies interpolation if and only if
	$d < 2r$ and $(d,g,r) \notin \left\{ (5,2,3), (6,2,4), (7,2,5) \right\}$.

	The proof of this statement is not difficult given the above results.
	When, $d < 2r$, the statement follows from
	\autoref{thm:NaskoEricDavid}.
	When $d = 2r$, it follows from 
\cite[p.\ 108]{stevens:deformations-of-singularities}.
Finally, to see that Castelnuovo curves of degree $d < 2r$, do not
satisfy weak interpolation, note that such a curve lies on a surface of minimal degree. However, if weak interpolation is equivalent to the Castelnuovo
curve passing through $n$ points, a dimension count shows that there will be such surface 
of minimal degree passing through $n$ points, and so there can be no
such Castelnuovo curve.
\end{example}

To summarize the above example, canonical curves approximately
``form the boundary'' between Castelnuovo curves satisfying
interpolation and 
Castelnuovo curves not satisfying interpolation.

\subsubsection{Higher dimensional varieties: Varieties of minimal degree}

Recent work of the first author \cite{landesman:interpolation-of-varieties-of-minimal-degree}, establishes interpolation for all varieties of minimal degree. Recall
that a variety of dimension $k$ and degree $d$ in $\bp^n$
is of minimal degree
if it is not contained in a hyperplane and
$d + k = n - 1$.
By \cite[Theorem 1]{eisenbudH:on-varieties-of-minimal-degree}, an irreducible variety is of minimal degree if and only if
it is a degree 2 hypersurface, the $2$-Veronese in $\bp^5$
or a rational normal scroll.

\begin{theorem}[Landesman, \cite{landesman:interpolation-of-varieties-of-minimal-degree}]\label{thm:stronginterpscroll}
Smooth varieties of minimal degree satisfy interpolation.
\end{theorem}
\begin{remark}
	\label{remark:}
	Parts of ~\autoref{thm:stronginterpscroll}
	have been previously established. For example,
	the dimension $1$ case is covered in \ref{sssec:castelnuovos-lemma}.
	The Veronese surface was shown to satisfy interpolation
	in \cite[Theorem 19]{coble:associated-sets-of-points},
	see ~\autoref{ssec:2-veronese-interpolation}
	for a more detailed description of this proof.
	It was already established that $2$-dimensional
	scrolls satisfy interpolation in
	Coskun's thesis ~\cite[Example, p.\ 2]{coskun:degenerations-of-surface-scrolls}, and furthermore,
	Coskun gives a method for computing the number of scrolls
	meeting a specified collection of general linear spaces.
	Finally, weak interpolation was established for
	scrolls of degree $d$ and dimension $k$ with
	$d \geq 2k - 1$ in \cite[Theorem 4.5]{eisenbudP:the-projective-geometry-of-the-gale-transform}.
\end{remark}

\subsection{Approaches to interpolation}
There are at least three approaches to solving interpolation problems.  

The first approach is to directly construct a variety $[Y] \in \sch$ meeting the specified constructions.  This method is quite ad hoc: For one, we would need ways of constructing varieties in projective space.  Our ability to do so is very limited and always involves special features of the variety. For examples of this approach, see
\autoref{sec:degree-5}, \autoref{sec:degree-6}, and \autoref{sec:degree-8-type-0}.

The second standard approach is via specialization
and degeneration. In this approach,
we specialize the points to a configuration
for which it is easy to see there is an isolated point
of $\sch$ containing such a configuration.
Often, although not always, the isolated point of $\sch$
corresponds to
a singular variety.
Finding singular varieties may often be easier than
finding smooth ones, particularly if those singular
varieties have multiple components, because we may
be able to separately interpolate each of the components
through two complementary subsets of the points.
An instance of specialization, although in a slightly different context, can be found in
\autoref{ssec:existence-of-singular-triads}. Here is
a simpler example:

\begin{example}\label{example:degeneration-rnc}
	\label{example:}
	A simple example of specialization
and degeneration can be seen in proving
that there is a twisted cubic curve
through $6$ general
points in $\bp^3$. Start by specializing four of the five
points to lie in a plane. There are no smooth twisted
cubics through such a collection of points. However,
there is a singular twisted cubic, realized
as the union of a line and a plane conic through such a collection of points.

To see there is such a singular twisted cubic, note
that if we draw a line through the two non-planar points,
it will intersect the plane containing the four points
at a fifth point $p$. There will then be a unique conic
through $p$ and the $4$ planar points. The union of this
line and conic is a degenerate twisted cubic.
Omitting several technical details, this
curve ends up being isolated among curves in the irreducible component
of the Hilbert scheme of twisted cubics through this
collection of points, and hence twisted cubics
satisfy interpolation.
\end{example}

The third approach is via association, see
~\autoref{section:association} for more details on
what this means. The general picture
is that association determines a natural way of
identifying a set of $t$ points in $\bp^a$ with a collection of $t$ points
in $\bp^b$, up to the action of $\PGL_{a+1}$ on the first
and $\PGL_{b+1}$ on the second. 
Then, if one can find a certain variety through the $t$ points
in $\bp^a$, one may be able to use association to find
the desired variety through the $t$ associated points in
$\bp^b$. For an example of this approach, see
\autoref{ssec:2-veronese-interpolation} and \autoref{sec:degree-7-8-9}.

\subsection{Main results of this paper} 

Recall that a del Pezzo surface, embedded in $\bp^n$,
is a surface with ample anticanonical bundle, embedded
by the complete linear system of its
anticanonical bundle. All del Pezzo surfaces have
degree $d$ in $\bp^d$, and
all linearly normal smooth surfaces of degree $d$ in $\bp^d$ are del Pezzo
surfaces by
\cite[Theorem 2.5]{coskun:the-enumerative-geometry-of-del-Pezzo-surfaces-via-degenerations}.
We also know the dimension of the component of the Hilbert
scheme containing a del Pezzo surface from \cite[Lemma 2.3]{coskun:the-enumerative-geometry-of-del-Pezzo-surfaces-via-degenerations}, as given
in \autoref{table:del-Pezzo-conditions}, and that all del Pezzo
surfaces have $H^1(X, N_{X/\bp^n}) = 0$, by \cite[Lemma 5.7]{coskun:the-enumerative-geometry-of-del-Pezzo-surfaces-via-degenerations}.

Assuming the remainder of the paper, we now restate and prove
our main result.

\main*
\begin{proof}
Recall that there is a unique component of the Hilbert
scheme of del Pezzo surfaces in
degrees $3, 4, 5, 6, 7, 9$, and there are two in degree $8$.
One component in degree $8$, which we call type $0$,
has general member abstractly isomorphic to $\mathbb F_0 \cong \bp^1 \times \bp^1$.
The other component in degree $8$, which we call type $1$,
has general member abstractly isomorphic to $\mathbb F_1$.
The cases of degree $3$ and $4$ surfaces hold by
\autoref{lemma:balanced-complete-intersection}.
The case of degree $5$ del Pezzo surfaces is
\autoref{thm:quinticDPinterpolation}.
The case of degree $6$ del Pezzo surfaces is
\autoref{thm:sexticDPweakinterpolation}.
The case of degree $8$, type $0$
surfaces is ~\autoref{thm:weakinterpolationP1xP1}.
Finally, the three remaining cases of del Pezzo surfaces in
degrees $7,8,9$ are 
\autoref{corollary:degree-8-type-1-interpolation},
\autoref{corollary:degree-7-interpolation},
and
~\autoref{thm:3veroneseexistence},
respectively.
\end{proof}

\begin{table}
\begin{tabular}
	{| c || c | c | c |}
	\hline
	Degree & Dimension & Number of Points & Additional Linear Space Dimension, If Any \\
	\hline
	\hline

	$3$ & $19$ & $19$ & None \\	\hline

	$4$ & $26$ & $13$ & None \\	\hline
	
	$5$ & $35$ & $11$ & $1$ \\	\hline
	
	$6$ & $46$ & $11$ & $2$ \\	\hline
	
	$7$ & $59$ & $11$ & $1$ \\	\hline
	
	$8$, type 0 & $74$ & $12$ & $4$ \\	\hline
	
	$8$, type 1 & $74$ & $12$ & $4$ \\	\hline
	
	$9$ & $91$ & $13$ & None \\	\hline
\end{tabular}
\vspace{.5cm}
\caption{Conditions for del Pezzo surfaces to satisfy interpolation.
Type $0$ refers to the component of the Hilbert scheme whose
general member is a degree $8$ del Pezzo surface, isomorphic to $\mathbb F_0$,
(this also includes, in its closure, del Pezzo surfaces abstractly
isomorphic to $\mathbb F_2$,)
while type $1$ refers to those isomorphic to $\mathbb F_1$.
The dimension counts are proven in \cite[Lemma 2.3]{coskun:the-enumerative-geometry-of-del-Pezzo-surfaces-via-degenerations}.
}
\label{table:del-Pezzo-conditions}
\end{table}

\ssec{Organization of Paper}

The remainder of this paper is structured as follows:
We also
include a proof of the elementary fact that {\sl balanced} complete intersections satisfy interpolation.
In \autoref{sec:degree-5}, \autoref{sec:degree-6}, and \autoref{sec:degree-8-type-0}, we show del Pezzo surfaces of degree $5, 6$, and degree $8$, type $0$,
respectively,
satisfy weak interpolation. Our approach for
surfaces of degree $5,6$ and the degree $8$, type $0$ del Pezzo
surfaces is to find surfaces
through a collection of points by first finding a curve or threefold
containing the points and then a surface containing the curve or contained
in the threefold.
In \autoref{section:association}, we recall the technique
of association, in preparation for \autoref{sec:degree-7-8-9},
where we use association to prove weak interpolation of
del Pezzo surfaces of degree $7$, degree $8$, type 1, and degree $9$.  The degree $9$ del Pezzo surface case is by far the most technically challenging case in the paper. 
We were led to the approach of association after reading Coble's remarkable paper ``Associated Sets of Points" \cite{coble:associated-sets-of-points}. We discuss further questions and open problems in \autoref{sec:questions}.
Finally, in \autoref{sec:interpolation-in-general} we prove many distinct formulations of interpolation are equivalent. 

\subsection{Notation and conventions} We work over an algebraically closed field $k$ of characteristic zero, unless otherwise stated. We freely use the language of line bundles, divisor classes, and linear systems. When $V$ is a vector space of dimension $d$,
we sometimes write it as $V^d$ to indicate its dimension.

\subsection{Acknowledgments} We would like to thank Francesco Cavazzani, Izzet Coskun, Igor Dolgachev, Joe Harris, Brendan Hassett, J\'anos Koll\'ar,
Peter Landesman, Eric Larson, Rahul Pandharipande, and Adrian Zahariuc, and several anonymous referees for helpful conversations and correspondence.

\section{Degree 5 del Pezzos}

\label{sec:degree-5}
\begin{theorem}
	\label{thm:quinticDPinterpolation}
	Quintic del Pezzo surfaces satisfy
	weak interpolation.
\end{theorem}
\begin{proof}
	By \autoref{table:del-Pezzo-conditions}	it suffices to show quintic del Pezzo surfaces pass through $11$ points.
	
	Start by choosing $11$ points.
	Since degree $3$, dimension $3$	scrolls satisfy interpolation,
	by \autoref{thm:stronginterpscroll},
	there is such a scroll through any $12$ general points.
	Equivalently, there is a two dimensional family of scrolls
	through $11$ points, which sweeps out all of $\bp^5$.
	In any scroll in this two dimensional family, we will show
	there is a quintic del Pezzo.

	First, start with a scroll $X$ containing the $11$ points.
	Since $X$ is projectively normal and its ideal is defined by
	$3$ quadrics, $h^0(X, \sco_X(2)) = 21 - 3 = 18.$
	Therefore, if we let $P$ be a ruling two plane of $X$,
	since $h^0(P, \sco_P(2)) = 6$, there will be an $18 - 6 = 12$
	dimensional space of quadrics on $X$ vanishing on $P$.
	Therefore, there will be a $12 - 11 = 1$ dimensional
	space of quadrics vanishing on $P$ and containing the $11$ points.
	However, the intersection of any such quadric with $X$
	is the union of $P$ and a quintic del Pezzo surface. 
	Therefore, we have produced a two dimensional family of
	quintic del Pezzo surfaces containing the $11$ points.
	\end{proof}

\begin{remark}
	\label{remark:}
	Another way to prove weak interpolation of quintic del Pezzo
	surfaces uses curves instead of threefolds.
	Specifically, by \cite[Corollary 6]{stevens:on-the-number-of-points-determining-a-canonical-curve}, every genus $6$ canonical curve passes through
	$11$ general points. Then, since there is a quintic del Pezzo
	surface containing any genus 6 canonical curve,
	as proved in, among other places,
	\cite[5.8]{arbarelloH:canonical-curves-and-quadrics-of-rank-4}.
	Then, because there is a canonical curve containing these points
	and a quintic del Pezzo containing the canonical curve,
	there is a quintic del Pezzo containing these points.
\end{remark}

\section{Degree $6$ del Pezzos}
\label{sec:degree-6}
By \autoref{table:del-Pezzo-conditions}, weak interpolation for sextic del Pezzos amounts to showing that through $11$ general points $\Gamma_{11} \subset \bp^{6}$ there passes a sextic del Pezzo.

\begin{lemma}\label{lem:genus3degree9}
Through $11$ general points $\Gamma_{11} \subset \bp^{6}$ there passes a smooth degree $9$, genus $3$ curve. 
\end{lemma}

\begin{proof}
This is a special case of \autoref{thm:NaskoEricDavid}.
\end{proof}

Starting from a curve $C \subset \bp^{6}$ as in \autoref{lem:genus3degree9}, we can ``build'' a sextic del Pezzo surface containing $C$.

\begin{lemma}\label{lem:p+q+r}
Let $D$ be a general degree $9$ divisor class on a genus $3$ curve $C$. Then there exists a unique degree three effective divisor $P+Q+R$ such that $D \sim 3K_{C} - (P+Q+R)$.
\end{lemma} 

\begin{proof}
In general, if $X$ is a smooth genus $g$ curve, the natural map \[J \from \Sym^{g}C \to \Pic^{g}C\] is a birational morphism.

In our setting, if $D$ is a general degree $9$ divisor class, $3K_C - D$ will be a general degree $3$ divisor class, and therefore can be represented by a unique degree three divisor class $P+Q+R$. Of course, by Riemann-Roch,
every degree three divisor class is effective.
\end{proof}

\begin{lemma}\label{lem:sexticcontainingC}
	Let $\Gamma_{11} \subset \bp^{6}$ be general, and let $C$ be general among the degree $9$, genus $3$ curves containing $\Gamma_{11}$.  Then there is a smooth sextic del Pezzo surface containing $C$.
\end{lemma}

\begin{proof}
Since $\Gamma_{11}$ are chosen generally, we have that a general $C$ containing them is not hyperelliptic.
So, we may embed $C \subset \bp^{2}$ via its canonical series $|K_{C}|$. The linear system $|3K_{C} - (P + Q + R)|$ on $C$ is cut out by plane cubics passing through the three points $P + Q + R$. Under the generality conditions, we can assume $P,Q,R$ are not collinear in $\bp^{2}$.

The linear system of plane cubics through three noncollinear points maps $\bp^{2}$ birationally to a smooth sextic del Pezzo surface in $\bp^{6}$.
\end{proof}

\begin{theorem}\label{thm:sexticDPweakinterpolation}
Sextic del Pezzo surfaces satisfy weak interpolation.
\end{theorem}

\begin{proof}
To show a sextic del Pezzo satisfies interpolation,
by \autoref{table:del-Pezzo-conditions}, it suffices to show it passes through $11$ general points.
By ~\autoref{lem:genus3degree9}, there is a degree $9$ genus $3$
curve through $11$ general points in $\bp^6$. By \autoref{lem:sexticcontainingC},
there is a sextic del Pezzo containing a degree $3$ genus $9$ curve in
$\bp^6$.
\end{proof}

\section{The degree $8$, type 0 del Pezzos}
\label{sec:degree-8-type-0}
Next we consider $\bp^{1} \times \bp^{1} \subset \bp^{8}$ embedded by the linear system of $(2,2)$-curves. To prove weak interpolation, by \autoref{table:del-Pezzo-conditions}, we want to show there is such a surface passing through $12$ general points $\Gamma_{12} \subset \bp^{8}$.  As in the sextic del Pezzo case, we will again ``build'' a surface starting from a curve.

\begin{lemma}\label{lem:genus2degree10}
Through $12$ general points $\Gamma_{12} \subset \bp^{8}$ there passes a smooth genus $2$ curve of degree $10$.
\end{lemma}

\begin{proof}
This is a special case of \autoref{thm:NaskoEricDavid}.
\end{proof}

\begin{lemma}\label{lem:degree10divisors}
A general degree $5$ divisor class $D$ on a smooth genus $2$ curve may be written uniquely as $K_{C} + A$, where $A$ is a basepoint free degree $3$ divisor class.  A general degree $10$ divisor class $E$ can be expressed as $2D$ for $2^{4}$ distinct degree $5$ divisor classes $D$.
\end{lemma}

\begin{proof}
Similar to \autoref{lem:p+q+r}. We leave the details to the reader.
\end{proof}

\begin{lemma}\label{lem:buildP1P1}
The general genus $2$, degree $10$ curve $C \subset \bp^{8}$ is contained in a $\bp^{1} \times \bp^{1}$ embedded via the linear system of $(2,2)$ curves.
\end{lemma}

\begin{proof}
Let $H$ denote the degree $10$ hyperplane divisor class on $C \subset \bp^{8}$. Write $H = 2D$ for some degree five divisor class $D$, and write $D \sim K_{C} + A$ for a unique degree $3$ divisor class $A$. By generality assumptions, $A$ is basepoint free, and we obtain a map \[f \from C \to \bp^{1} \times \bp^{1}\] given by the pair of series $(|K_{C}|, |A|)$. This map embeds $C$ as a $(2,3)$ curve.

The linear system $|(2,2)|$ on this $\bp^{1} \times \bp^{1}$ restricts to the complete linear system $2D$ on $C$, and therefore induces the original embedding $C \subset \bp^{8}$.  The image of $\bp^{1} \times \bp^{1}$ under the system $|(2,2)|$ is therefore the surface we desire.
\end{proof} 

\begin{theorem}\label{thm:weakinterpolationP1xP1}
$\bp^{1} \times \bp^{1} \subset \bp^{8}$ embedded via the linear system of $(2,2)$ curves satisfies weak interpolation.
\end{theorem}

\begin{proof}
	This follows by combining \autoref{lem:genus2degree10} and \autoref{lem:buildP1P1}.
\end{proof}

\begin{remark}
An interesting feature of this solution to our interpolation problem is that the surfaces we've constructed through the $12$ general points $\Gamma_{12}$ are in fact {\sl special} among the two dimensional family of surfaces passing through these points.  Indeed, the set $\Gamma_{12}$ is contained in a $(2,3)$ curve on the surfaces we've constructed, but a general set of twelve points on $\bp^{1}\times \bp^{1}$ does not lie on any $(2,3)$ curve!
\end{remark}

\section{Interlude: Association} 
\label{section:association}

This section is meant to provide the reader with basic familiarity with {\sl association}, also known as the {\sl Gale transform}.  Association will be a recurring tool in the rest of the paper. We closely follow the exposition in \cite{eisenbudP:the-projective-geometry-of-the-gale-transform}.

\subsection{Preliminaries} Throughout this section, we let $\Gamma$ be a Gorenstein scheme, finite over $k$ of length $\gamma = r + s + 2$, $L$ an invertible sheaf on $\Gamma$, and $V \subset H^{0}(\Gamma, L)$ a vector space of dimension $r+1$.  In practice, $\Gamma$ will be given as embedded in projective space $\bp^{r}$, $L$ will be $\so_{\Gamma}(1)$, and $V$ will be the image of the restriction map \[H^{0}(\bp^{r}, \so_{\bp^{r}}(1) ) \to H^{0}(\Gamma, \so_{\Gamma}(1)).\] 
For brevity, we will often refer to the data of the pair $(V, L)$ as a {\sl linear system} on $\Gamma$. For clarity, we will sometimes put subscripts on $\Gamma$ emphasizing the number of points.

The Gorenstein hypothesis on $\Gamma$ says that the dualizing sheaf $\omega_{\Gamma}$ is a line bundle, and furthermore Serre duality holds: There is a trace map $t \from H^{0}(\Gamma, \omega_{\Gamma}) \to k$, and for any line bundle $L$ the {\sl trace pairing}  \[H^{0}(\Gamma, L) \otimes H^{0}(\Gamma, {\dual L} \otimes \omega_{\Gamma}) \to k\] is nondegenerate.  

Therefore if $V$ is a $r+1$ dimensional subspace of $H^{0}(\Gamma, L)$, we obtain a natural $s+1$-dimensional subspace \[V^{\perp} \subset H^{0}(\Gamma, {\dual L} \otimes \omega_{\Gamma}),\] namely the orthogonal complement of $V$ under the trace pairing.

\begin{definition}\label{def:associatedseries}
Let $\Gamma$ be a length $\gamma$ Gorenstein scheme over $k$, and let $(V, L)$ be a linear system on $\Gamma$. Then we say $(V^{\perp}, {\dual L} \otimes \omega_{\Gamma})$ is the {\sl associated linear system} of $(V,L)$.
\end{definition} 

Notice that association provides a correspondence between {\sl vector spaces} $V \leftrightarrow V^{\perp}$, and not between vector spaces with chosen bases. Geometrically this means association provides a bijection between the $\PGL_{r+1}(k)$--orbits of Gorenstein $\Gamma \subset \bp^{r}$ (in general linear position) and $\PGL_{s+1}(k)$--orbits of Gorenstein $\Gamma \subset \bp^{s}$ (in general linear position).  Given this, in the future when we refer to ``the associated set,'' we really mean the $\PGL_{s+1}(k)$--orbit.  Moreover, it is known that association provides an isomorphism of GIT quotients \[(\bp^{r})^{\gamma}//\PGL_{r+1}(k) \xrightarrow{\sim} (\bp^{s})^{\gamma}//\PGL_{s+1}(k),\] and therefore takes general subsets to general subsets.

\subsection{Inducing association from an ambient linear system} Association is a very algebraic construction. Therefore, it is interesting to find geometric constructions which ``induce'' association for a set $\Gamma \subset \bp^{r}$.  To see many examples of the geometry underlying association, we refer to \cite{eisenbudP:the-projective-geometry-of-the-gale-transform}. 

In \cite[p.\ 2]{coble:associated-sets-of-points}, Coble asks, in less modern language, whether there exists a linear system $W^{s+1} \subset H^{0}(\bp^{r}, \so(d))$ whose base locus is disjoint from $\Gamma$, and which restricts on $\Gamma$ to the associated linear system.

A linear system $W^{s+1} \subset H^{0}(\bp^{r}, \so(d))$ yields a rational map \[\phi_{W} \from \bp^{r} \dashrightarrow \bp^{s}.\] 

\begin{definition}
Let $\Gamma \subset \bp^{r}$ be a Gorenstein scheme of degree $\gamma = r+s+2$. 
An {\sl ambient linear system} is any vector space $V \subset  H^{0}(\bp^{r}, \so(d))$.
An ambient linear system $W^{s+1} \subset  H^{0}(\bp^{r}, \so(d))$  {\sl induces association for} $\Gamma$ if its base locus is disjoint from $\Gamma$ and if the image $\phi_{W}(\Gamma) \subset \bp^{s}$ is the associated set of $\Gamma$.  
\end{definition}
  It is important to note, as Coble does, that an ambient system inducing association won't be unique in general.

  When association is induced from an ambient system, we automatically get a variety $\phi_{W}(\bp^{r}) \subset \bp^{s}$ containing $\Gamma$. Our task is ultimately to find an ambient linear system $W$ which induces association  for $\Gamma$, and such that the image $\phi_{W}(\bp^{r})$, (by this, we mean the image of the resolution of $\phi_{W}$) is a prescribed type of variety, e.g. Veronese images, del Pezzo surfaces, etc. 

\subsection{Goppa's theorem} Goppa's theorem is frequently useful when looking for ambient systems inducing association.  

\begin{theorem}[Goppa's Theorem]\label{thm:goppa}
Let $f \from B \to \bp^{r}$ be a map from a smooth curve given by a nonspecial, complete linear system $|H|$. Let $\Gamma \subset B$ be a scheme of length $\gamma = r+s+2$.  Then association for $\Gamma$ is induced by the restriction of the linear system $|K_{B}+\Gamma-H|$ to $\Gamma$.
\end{theorem} 

In practice, we will typically find a curve $B \subset \bp^{r}$ passing through $\Gamma$, and will try to induce association by realizing the linear system $|K_{B} + \Gamma - H|$ on $B$ via an ambient system on $\bp^{r}$.

\subsection{The $2$-Veronese surfaces through $9$ general points in $\bp^{5}$}
\label{ssec:2-veronese-interpolation}
We conclude this section with a result going back to Coble \cite[Theorem 19]{coble:associated-sets-of-points} and more rigorously explored in Dolgachev \cite[Theorems 5.2 and 5.6]{dolgachev:on-certain-families-of-elliptic-curves-in-projective-space} showing there are four Veronese surfaces in $\bp^{5}$ containing $9$ general points (in characteristic not equal to $2$). 

The following result follows without too much work from \cite{dolgachev:on-certain-families-of-elliptic-curves-in-projective-space} although it isn't explicitly stated there. 

The key to finding $2$-Veronese surfaces through $9$ points is to find a genus $1$ curve through the $9$ points,
and then find a $2$-Veronese surface containing that curve. We start off by understanding $2$-Veronese
surfaces containing a genus $1$ curve.

\begin{proposition}
  \label{proposition:veronese-surfaces-through-genus-1-curves}
  Let $k$ be an algebraically closed field and let $E \subset \bp^5_k$
  be a general genus $1$ curve, embedded by a complete linear series of degree
  $6$.
  If $\ch k \neq 2$, there are four $2$-Veronese surfaces
  containing $E$ and if $\ch k = 2$, there are two $2$-Veronese
  surfaces containing $E$.
\end{proposition}
\begin{remark}
  \label{remark:}
  It is shown in
     \cite[Theorem 5.6]{dolgachev:on-certain-families-of-elliptic-curves-in-projective-space} that there are exactly four $2$-Veronese surfaces containing
  a given genus 1 curve of degree $6$ in $\bp^5$
over a field of characteristic $0$.
  However, the proof given there does not make it completely clear
  why there is a unique $2$-Veronese surface through general $E$
  corresponding to each chosen square root of the line bundle embedding
  $E$. Therefore, we now repeat the proof in more detail,
  and generalize it to all characteristics.
  In fact, the proof shows there are $4$ such $2$-Veronese surfaces
  for all $E$ if $\ch k \neq 2$, there are $2$ such $2$-Veronese surfaces
  if $\ch k = 2$ and $E$ corresponds to a non-supersingular elliptic curve after choosing a point,
  and there is $1$ such $2$-Veronese if $\ch k = 2$ and $E$ corresponds to a supersingular elliptic
  curve after choosing a point.
\end{remark}
\begin{proof}
  Say $E \ra \bp^5$ is given by the invertible sheaf $\scl$. For any degree three invertible sheaf
$\scm$ with $\scm^{\otimes 2} \cong \scl$, we can map $E \ra \bp^2$ using $\scm$. Then, the composition of
$E \ra \bp^2$ with the $2$-Veronese map $\bp^2 \ra \bp^5$ will send $E$ to $\bp^5$ by $\scl$ and so we have
constructed a $2$-Veronese surface containing $E$. Since there are two such sheaves $\scm$
in characteristic $2$ for a general $E$ and four in all other characteristics (since a non supersingular elliptic curve has
two $2$ torsion points in characteristic $2$ and every genus $1$ curve has four such points in all other characteristics), it suffices to show these are the only 
$2$-Veronese surfaces containing $E$.
That is, we only need show that for each square root $\scm$ of $\scl$, there is a unique $2$-Veronese surface
$X \cong \bp^2$ containing $E$ so that the map $E \ra \bp^2$ is given by a basis for the global sections of $\scm$.
 
  First, note that if an automorphism fixes $E$ pointwise then it fixes all of $\bp^5$.
  This holds because $E$ spans $\bp^5$, and so a linear automorphism fixing $E$ pointwise would
also fix a basis for the vector space $H^0(\sco_E(1))$ which satisfies $\bp H^0(\sco_E(1)) \cong \bp^5$. Hence, such an automorphism would 
fix all of $\bp^5$.

 Suppose we have two $2$-Veronese surfaces $X$ and $X'$ containing
  $E$ so that $E$ 
  we have a map $\phi_1: E \ra X$ and $\phi_2 : E \ra X'$ so that
  both maps $\phi_1$ and $\phi_2$ are given by the same degree $3$ invertible
sheaf $\scm$, together with a choice of basis for $H^0(E, \scm)$.
  We will show that there exists an automorphism $\phi:\bp^5 \ra \bp^5$
  fixing $E$ pointwise and sending $X$ to $X'$. Since any automorphism of $\bp^5$
  fixing $E$ pointwise is the identity, this would imply $X = X'$, and would complete the proof.
 
  First, we show there is an automorphism $\phi: \bp^5 \ra \bp^5$
  fixing $E$ as a set and taking
  $X$ to $X'$.
  We know there is an automorphism $\psi:\bp^5 \ra \bp^5$ with $\psi(X) = X'$.
  Say $\psi$ sends the curve $E \subset X$
  to some curve $E' := \psi(E) \subset X'$.
  Next, by our assumption that $E$ and $E'$ are two curves on $X'$ both given by
  global sections associated to the same invertible sheaf $\scm$,
  there is some automorphism of $\psi': X' \ra X'$ with $\psi'(E') = E$.
  Thus, taking $\phi := \psi' \circ \psi$, we see
  $\phi(X) = X'$ and $\phi(E) = E'$ as sets.
  If we could arrange for $\phi|_E = \id$, we would be done, as then $\phi = \id$.

  Hence, it suffices to show that $\phi|_E$ is an automorphism of $E$ fixing both 
  $X$ and $X'$.    

  Let $A(E, \scm)$ denote the automorphisms $\pi:E \ra E$ with $\pi^* \scm \cong \scm$.
  Note that we have an exact sequence
\begin{equation}
  \nonumber
  \label{equation:}
  \begin{tikzcd} 
    0 \ar{r} & E[3] \ar {r} & A(E, \scm) \ar{r} & \bz/2 \ar{r} & 0 
  \end{tikzcd}\end{equation}
  where the generator of the quotient $\bz/2$ is the hyperelliptic involution and the
  subset $E[3]$ is a torsor over the $3$ torsion of $E$ with any given
  choice of origin.
  In particular, if we choose a point $p$ so that $\scm \cong \sco_E(3p)$,
  we have that $E[3]$ is precisely translation by $6$-torsion.

  It suffices to show that any element of $A(E, \scm)$
  fixes the $2$-Veronese surface we constructed above corresponding
  to $\scm$.

  But, if we view $E \ra \bp^2$ by a complete linear system corresponding
  to $\scm$, the automorphisms
  $A(E, \scm)$ are precisely the automorphisms of $\bp^2$
  fixing $E \subset \bp^2$ as a set.
  These automorphisms of $\bp^2$ extend to automorphisms on $\bp^5$
  with $\bp^2 \ra \bp^5$ embedded via the $2$-Veronese map.
  Therefore, they also fix the $2$-Veronese surface, as desired.
\end{proof}

\begin{theorem}
  \label{theorem:counting-veronese-interpolation}
  Through 9 general points in $\bp^5_k$ there exist four
  $2$-Veronese surfaces $\bp^2 \ra \bp^5$ if $k$
  is an algebraically closed field with $\ch k \neq 2$
  and two $2$-Veronese surfaces $\bp^2 \ra \bp^5$
  if $k$ is an algebraically closed field with $\ch k = 2$.
  In particular, the $2$-Veronese surface satisfies interpolation.
\end{theorem}
\begin{proof}
  Fix $9$ general points $p_1, \ldots, p_9 \in \bp^5$.
  First, by \cite[Theorem 5.2]{dolgachev:on-certain-families-of-elliptic-curves-in-projective-space}, there is a unique genus 1 curve embedded by
  a complete linear series through 9 general points in $\bp^5$.
  Call this curve $E$.
  Next, by \autoref{proposition:veronese-surfaces-through-genus-1-curves},
  there are four $2$-Veronese surfaces containing $E$
  if $\ch k \neq 2$ and two $2$-Veronese surfaces containing $E$
  if $\ch k = 2$.
  To complete the proof, 
  it suffices to show that every $2$-Veronese surface containing
  $p_1, \ldots, p_9$
  also contains $E$.
  Consider such a $2$-Veronese surface $X \subset \bp^5_k$ containing
  $p_1, \ldots, p_9$.
  Choosing an isomorphism $\phi: \bp^2_k \cong X$,
  we have nine points $q_1, \ldots, q_9$ on $\bp^2$
  so that $\phi(q_i) = p_i$.

  Then, since $p_1, \ldots, p_9$ were general on $\bp^5_k$,
  we have that $q_1, \ldots, q_9$ are general on $\bp^2_k$,
  and so there is a degree $3$ genus 1 curve $C$ passing through
  $q_1, \ldots, q_9$ on $\bp^2_k$.
  The image of $\phi(C)\subset X$ is a degree $6$ genus $1$ curve
  containing $p_1, \ldots, p_9$.
  Since $E$ is the unique genus $1$ degree $6$ curve $E$ containing
  $p_1, \ldots, p_9$, we must have $\phi(C) \cong E$,
  and therefore $E \subset X$.
\end{proof}

\begin{remark}
Starting with a general $\Gamma_{9} \subset \bp^{5}$, we obtain
an associated set $A(\Gamma_{9}) \subset \bp^{2}$, a set of nine general points in the plane. It is tempting to re-embed this $\bp^{2}$ via the complete system of conics and hope that the image of the set $A(\Gamma_{9})$ is projectively equivalent to $\Gamma_{9}$.  However, this is not the case -- the system of conics in $\bp^{2}$ does {\sl not} induce association for a set of nine general points.
\end{remark}

\section{Degrees $9,8$ and $7$}
\label{sec:degree-7-8-9}
This section establishes weak interpolation for degree $9$ Del Pezzo surfaces, which are $3$-Veronese images of $\bp^{2}$ in $\bp^{9}$.  As we will see
in \autoref{ssec:degrees-7-and-8},
weak interpolation for degree $8$, type 1, and degree $7$ Del Pezzo surfaces immediately follow from the proof for degree $9$. We will also see in
\autoref{ssec:degrees-7-and-8} that tricanonical genus $3$ curves satisfy
interpolation.

\subsection{Results}
The main result of this section is:

\begin{theorem}[Existence]\label{thm:3veroneseexistence}
Let $\Gamma \subset \bp^{9}$ be thirteen general points. Then there exists a $3$-Veronese surface  containing $\Gamma$.
\end{theorem}
\begin{proof}[Proof assuming \autoref{thm:bijectionTriadsVeronese} and \autoref{theorem:DegenerateConfigurationInClosure}]
	By \autoref{thm:bijectionTriadsVeronese}, for a general
	$\Gamma_{13} \in \Hilb_{13} \bp^2$, there is a bijection
	between singular triads for $\Gamma_{13}$ (defined below in \autoref{def:singulartriad})
	and $3$-Veronese surfaces containing the associated set $A(\Gamma_{13}) \subset \bp^{9}$. 
	By \autoref{theorem:DegenerateConfigurationInClosure},
	every such $\Gamma_{13}$ indeed possesses a singular triad.
\end{proof}

The essential tool used in proving \autoref{thm:3veroneseexistence} is association. 

Our next result relates the number of Veronese surfaces through $13$ general points to another, more tractable  enumerative problem.  Before stating it, we must make a definition. 

\begin{definition}\label{def:singulartriad}
Let $\Gamma \subset \bp^2$ be a general set of thirteen points in the plane. A subset $T = \{x,y,z\} \subset \bp^2 \setminus \Gamma$ of three distinct points is a {\sl singular triad} for $\Gamma$ if \[h^{0}(\bp^{2}, \so_{\bp^{2}}(5)\otimes \sci_{T}^{2}\sci_{\Gamma}) = 2.\]
\end{definition}

\begin{remark}
	In other words, $T = \{x,y,z\}$ is a singular triad for $\Gamma$ if there exists a pencil of quintic curves through $\Gamma$ and singular at $x, y,$ and $z$. A dimension count shows that we expect finitely many singular triads for a general set of thirteen points $\Gamma$, as is done in \autoref{lem:dimensionPhi}.
\end{remark}

Our second result is:

\begin{theorem}[Enumeration]\label{thm:3veronesenumber}
The number of $3$-Veronese surfaces through a general set of thirteen points in $\bp^{9}$ is equal to the number of singular triads for a general set of thirteen points in $\bp^{2}$.
\end{theorem}

\autoref{thm:3veronesenumber} points to an interesting enumerative problem on the Hilbert scheme ${\rm Hilb}_{3}(\bp^{2})$ of degree $3$, zero dimensional subschemes of the plane.  We discuss this problem at the end of the section, in 
\autoref{ssec:enumerating-singular-triads}.

\subsection{How singular triads arise}\label{section:howsingtriadsArise} For the benefit of the reader, we briefly explain how singular triads arise in the problem of enumerating $3$-Veronese surfaces through a general set $\Gamma_{13}$.  

Suppose $V_{3} \subset \bp^{9}$ is a $3$-Veronese containing $\Gamma_{13}$.   If we consider $V_{3}$ as isomorphic to $\bp^{2}$, it is tempting, as in the case of the $2$-Veroneses, to think that the linear system $|3H|$ on $\bp^{2}$ would induce association for $\Gamma_{13} \subset \bp^{2}$.  However, this turns out not to be the case. 

In $\bp^{2}$, there is a unique pencil of quartic curves $Q_{t} \subset \bp^{2}$, $t \in \bp^{1}$, containing $\Gamma_{13}$.  Assuming the configuration $\Gamma_{13}$ is general, the pencil $Q_{t}$ will have three remaining, noncollinear basepoints $\{p,q,r\}$.  Let \[\alpha_{\{p,q,r\}} \from \bp^{2} \dashrightarrow \bp^{2}\] be the Cremona transformation centered on the set $\{p,q,r\}$, and let $T = \{x,y,z\}$ be the exceptional set 
in the target $\bp^{2}$. Then $\alpha(Q_{t})$ is a pencil of quintic curves, singular at $T$ and containing $\alpha(\Gamma_{13})$ in its base locus.  In other words, $T$ forms a singular triad for $\alpha(\Gamma_{13})$. In the next section, we will show that the ambient system of sextics having triple points at $x,y$, and $z$ induces association for $\alpha(\Gamma_{13})$.  In other words, the ``naive'' system of cubics on the source $\bp^{2}$ induces association {\sl not} for $\Gamma_{13}$, but rather for $\alpha(\Gamma_{13})$.

\subsection{Inducing association from a singular triad} 

We begin with a lemma describing the base locus of a pencil through $13$
points, whose proof is straightforward.
\begin{lemma}\label{lem:quinticbaselocus}
	Assume $\Gamma_{13} \subset \bp^2$ is a general set of $13$ points, and suppose $T = \{x,y,z\}$ is a singular triad for $\Gamma_{13}$, i.e. there exists a pencil of quintics $Q_{t}$ through $\Gamma_{13}$, singular at $x,y,z$. Furthermore, assume that the general element of the pencil has a smooth genus $3$ normalization, and has ordinary nodes at $x,y,z$. In particular, this implies $T$ is not contained in a line.  
Then, the scheme theoretic base locus of the pencil $Q_{t}$ consists of $\Gamma_{13}$ and three length four schemes supported on $x,y$ and $z$.
\end{lemma}

\begin{proposition}\label{prop:6h-3x-3y-3z}
	In the setting of \autoref{lem:quinticbaselocus}, the ten dimensional vector space \[W := H^{0}(\bp^{2}, \sci_{T}^{3}(6)) \subset H^{0}(\bp^{2}, \so(6))\] consisting of sextics having triple points at $x,y$ and $z$ induces association for $\Gamma_{13}$.
\end{proposition}

\begin{proof}
	We use Goppa's theorem, \autoref{thm:goppa}. Pick a general quintic $Q$ in the pencil $Q_{t}$, and let $\nu \from \td{Q} \to Q$ denote the smooth genus $3$ normalization.
	Let $H$ denote the hyperplane divisor class on $\bp^{2}$. 
	Note that by degree considerations and Riemann-Roch, the divisor class $H$ is nonspecial on $\td{Q}$, and $\td{Q}$ is mapped 
		via the complete linear series $|H|$. We claim that the linear system $|K_{\td{Q}} + \Gamma_{13} - H|$ from Goppa's theorem is induced by sextic curves triple at $x,y$ and $z$.

Indeed, the canonical series $|K_{\td{Q}}|$ is cut out by the adjoint series consisting of conics passing through the nodes $x,y,$ and $z$.  By \autoref{lem:quinticbaselocus}, the divisor $\Gamma_{13}$ is cut out by a quintic singular at $x,y,z$.  Putting these together says that sextics having triple points at $x,y,$ and $z$ cut out divisors in the linear system $|K_{\td{Q}} + \Gamma_{13} - H|$ on $\td{Q}$. 

Finally, notice that there cannot be a sextic triple at $x,y$ and $z$ which also vanishes identically on $Q$ -- the residual curve would be a line containing $T$, but we are assuming $x,y,z$ are not collinear. Therefore, the system of sextics having triple points at $x,y,$ and $z$ cuts out the complete linear system $|K_{\td{Q}} + \Gamma_{13} - H|$. 
\end{proof}

\subsection{The bijection between singular triads and Veroneses} Let $\Gamma_{13} \subset \bp^{9}$ be thirteen general points, and let $A(\Gamma_{13}) \subset \bp^{2}$ denote the associated set. 

We have already seen in \autoref{section:howsingtriadsArise} that a $3$-Veronese $V_{3}$ containing $\Gamma_{13}$ arises from a singular triad $T$ for $A(\Gamma_{13})$. Let us now show that distinct triads provide distinct Veroneses.

\begin{proposition}\label{prop:distinctTriads}
Maintain the setting above. Distinct triads $T$ and $T'$ for $A(\Gamma_{13})$ give rise to distinct Veronese surfaces $V_{3}$ and $V'_{3}$ containing $\Gamma_{13}$. 
\end{proposition}

\begin{proof}
Let $W = H^{0}(\bp^{2}, \so_{\bp^{2}}(2)\otimes\sci_{T})$ and $W' = H^{0}(\bp^{2}, \so_{\bp^{2}}(2)\otimes\sci_{T'})$ be the vector spaces of conics passing through $T$ and $T'$ respectively.  

Denote by $\iota \from \bp^{2} \dashrightarrow \bp(W)$ and $\iota' \from \bp^{2} \dashrightarrow \bp(W')$ the Cremona maps associated to $W$ and $W'$. 

By \autoref{prop:6h-3x-3y-3z}, the vector spaces $\Sym^{3}W$ and $\Sym^{3}W'$ both induce association for $A(\Gamma_{13})$, so we identify them as the ten dimensional vector space $V$ giving the original embedding $\Gamma_{13} \subset \bp^{9}$. 

Let $\nu \from \bp(W) \hookrightarrow \bp^{9}$ and $\nu' \from \bp(W') \hookrightarrow \bp^{9}$ denote the respective Veronese maps. (Note that the target $\bp^{9}$ for the maps $\nu$ and $\nu'$ are the ``same'' given the previous paragraph.) 

The two Veronese surfaces $\bp(W)$ and $\bp(W')$ would be the same if and only if there existed an isomorphism $\alpha \from \bp(W) \to \bp(W')$ such that $\nu' \circ \alpha \circ \iota = \nu' \circ \iota'$ as rational maps from $\bp^{2}$ to $\bp^{9}$. 

But the indeterminacy locus of a rational map is determined by the map, and the indeterminacy locus of $\nu' \circ \alpha \circ \iota$ is $T$, whereas the indeterminacy locus of $\nu' \circ \iota'$ is $T'$. This completes the proof.
\end{proof}

\begin{theorem}\label{thm:bijectionTriadsVeronese}
Let $\Gamma_{13} \subset \bp^{9}$ be a general set of thirteen points.  Then the $3$-Veronese surfaces containing $\Gamma_{13}$ are in bijection with the singular triads for $A(\Gamma_{13}) \subset \bp^{2}$.
\end{theorem}

\begin{proof}
This follows immediately from \autoref{section:howsingtriadsArise} and \autoref{prop:distinctTriads}.
\end{proof}

\subsection{Existence of singular triads}
\label{ssec:existence-of-singular-triads}

\begin{definition}
Define $\Phi \subset \Hilb_{3}\bp^{2} \times \Hilb_{13}\bp^{2}$ to be the closure of the set of 
pairs 
$(\{x,y,z\}, \Gamma_{13}) \subset \Hilb_3\bp^2 \times \Hilb_{13} \bp^2$ for which $\left\{ x, y, z \right\}$ is disjoint from the support of $\Gamma_{13}$, and for which there exists a pencil of quintics singular at $x,y,z$ whose base locus is precisely $\left\{ x,y,z \right\} \cup \Gamma_{13}$. 
Define the projections
\begin{equation}
	\nonumber
	\begin{tikzcd}
		\qquad & \Phi \ar {ld}{\pi_1} \ar {rd}{\pi_2} & \\
		\Hilb_3 \bp^2 && \Hilb_{13} \bp^2.
	 \end{tikzcd}\end{equation}
\end{definition}

\begin{theorem}\label{theorem:DegenerateConfigurationInClosure} There exists a point $(\{x,y,z\}, A_{13}) \in \Phi$ which is isolated in its fiber under the second projection $\pi_{2} \from \Phi \to \Hilb_{13}\bp^{2}$. In particular, $\pi_{2}$ is dominant, and a general set $\Gamma_{13}$ possesses a singular triad.
\end{theorem}

The rest of the section is devoted to the proof of \autoref{theorem:DegenerateConfigurationInClosure}.  Before we proceed with the proof in Subsection \ref{sssec:degeneration-proof}, we set some notation and outline the idea of the proof in
\ref{sssec:degeneration-proof-idea}

\begin{definition}
	\label{definition:}
	Let $x_0, x_1, x_2$ denote three fixed non-collinear points in $\bp^2$ and set $l_{i,j} :=\overline{x_{i}x_{j}}$ forming the coordinate triangle. 

Let $X := \Bl_{\{x_{0},x_{1},x_{2}\}}\bp^{2}$, and let $E_{i}$ denote the exceptional divisor over $x_{i}$, $i = 1,2,3$. Set $L_{i,j}$ to be the proper transforms of the lines $l_{i,j} := \overline{x_i, x_j}$.   We let $H$ denote the hyperplane class on $\bp^{2}$ and its pullback on $X$.  By a {\sl line} in $X$, we mean an element of the linear system $|H|$ on $X$.
\end{definition}

\sssec{The idea of the Proof of \autoref{theorem:DegenerateConfigurationInClosure}.}
\label{sssec:degeneration-proof-idea}

In order to prove \autoref{theorem:DegenerateConfigurationInClosure}, we will
construct a particular set $[\Gamma_{13}] \in \Hilb_{13} \bp^2$ which
we will be able to see is isolated in its fiber under the map $\pi_2$.
The construction is as follows. 
Start by choosing a general line $M$ and a general point $p_{7}$
not on $M$. Then, choose points 
\begin{align*}
	p_1, p_2 &\in \ell_{0,1} \\
	p_3, p_4 &\in \ell_{0,2} \\
	p_5, p_6 &\in \ell_{1,2} \\
	p_8, p_9, p_{10} &\in M
\end{align*}
all general with respect to the above conditions.
We will then see that there is an element $(\left( x_0, y_0, z_0 \right), \Gamma_{13}) \in \Phi$ so that $p_1 \cup \cdots \cup p_{10} \subset \Gamma_{13}$,
and further that the remaining degree three scheme of $\Gamma_{13}$ is
supported on $M$. The hard part of the proof will be seeing that this
configuration lies in $\Phi$. This is done in \autoref{cor:closureofresidual}.
Once we know this configuration does lie in $\Phi$, it is not difficult
to see it is isolated. Since $\Gamma_{13}$ intersects $M$ with degree $6$,
every quintic containing $\Gamma_{13}$ must contain $M$. We are then looking
for a pencil of quartics with base locus containing $p_1, \ldots, p_6, p_7$
and having three additional singular nodes.
If the three singular nodes
do not lie on $M$, then this can only happen if the pencil of quartics
contains curves in its base locus. A case by case analysis shows
that if the three nodes are not collinear, the only possibility, up to permutation of the points, is that the base locus of this pencil of quartics is
$\ell_{0,1} \cup \ell_{0,2} \cup \ell_{1,2} \cup p_{7}$ and the moving part
of this pencil is the pencil of lines containing $p_{7}$. This
will be isolated in its fiber. 
Then, this means $\pi_2$ is dominant
because both varieties are irreducible and
$\dim \Phi = 26 = \dim \Hilb_{13}\bp^2$, as shown in \autoref{lem:dimensionPhi}. This concludes our sketch of the idea of the proof.

The proof of the following lemma is straightforward, and we omit its proof.
\begin{lemma}\label{lem:generalquinticpencil}
Let $\Gamma_{10} \subset X$ be ten general points. Then there is a unique pencil in the linear system $|5H - 2E_{1} - 2E_{2} - 2E_{3}|$ containing $\Gamma_{10}$  in its base locus.  Furthermore, the base locus of this pencil consists of the union of  $\Gamma_{10}$, and three residual points $\{a,b,c\} \subset S$ disjoint from $\Gamma_{10}$.
\end{lemma}

\begin{lemma}\label{lem:dimensionPhi}
$\Phi$ is $26$-dimensional.
\end{lemma} 
\begin{proof}
	First select three general points $\{x,y,z\}$ in $\bp^{2}$, giving $6$ dimensions. Using \autoref{lem:generalquinticpencil}, a general pencil of quintics singular at $\{x,y,z\}$ is determined by choosing ten general points to be in its base locus. The remaining three points of the base locus are determined
	by the initial choice of 10, by \autoref{lem:generalquinticpencil}. In total, we have that $\Phi$ is $26 = 6 + 2 \cdot 10$ dimensional.
\end{proof}

Let $\Gamma_{10}(t) = \{p_{1}(t), p_{2}(t), ... , p_{10}(t)\} \subset X \times \Delta$ be a family of ten points, parameterized by $\Delta := \Spec k[[t]]$, general among those with the following properties:

\begin{enumerate}
\item Over the generic point $\eta \in \Delta$, the points $p_{i}(\eta)$ are general in the sense of \autoref{lem:generalquinticpencil}. 

\item Over the special point $t=0$, the ten points $p_{i}(0)$ are situated as follows: 

\begin{enumerate}
\item $p_{1}(0), p_{2}(0)$ are general in $L_{0,1}$.
\item $p_{3}(0), p_{4}(0)$ are general in $L_{0,2}$.
\item $p_{5}(0), p_{6}(0)$ are general in $L_{1,2}$.
\item $p_{7}(0)$ is general in $X$.
\item $p_{8}(0), p_{9}(0), p_{10}(0)$ are general on a general line $M \subset X$.
\end{enumerate}
\end{enumerate}

By \autoref{lem:generalquinticpencil}, there are three residual points $\{a(\eta), b(\eta), c(\eta)\}$ defined by the ten points $\{p_{i}(\eta)\}_{i=1, ... , 10}$. We let $\{a(t), b(t), c(t)\}$ denote the closures of these points.  (Note: a base change may be required to say the three residual basepoints $\{p_{i}(\eta)\}_{i=1, ... , 10}$ are defined over $\Spec k((t))$. Performing such a base change does not affect the rest of the arguments.)

Now let $\scx$ be the threefold which is the blow up of $X \times \Delta$ at the union of the three curves $L_{i,j} \subset X \times \{0\}$, and let $\beta \from \scx \to X \times \Delta$ be the blow up map.  Let $f \from \scx \to \Delta$ denote the composition of $\beta$ with the projection onto the second factor of $X \times \Delta$.  $\scx_{\eta}$ and $\scx_{0}$ will denote the general and special fibers of $f$.  Note that $\scx_{\eta} = X_{\eta} := X \times \Spec k((t))$.

There are three exceptional divisors $F_{i,j}$ lying over the corresponding curves $L_{i,j} \subset X \times \{0\}$.  The map $f$ is a flat family of surfaces, with generic fiber $\scx_{\eta} = X_{\eta}$ and with special fiber $\scx_{0}$ a simple normal crossing union of four surfaces: 
the exceptional divisors $F_{i,j}$, and $X$.  Their incidence is as follows: The surfaces $F_{i,j}$ are pairwise disjoint and $F_{i,j} \cap X = L_{i,j}$.  

Each exceptional divisor $F_{i,j}$ is isomorphic to the Hirzebruch surface $\F_{1}$. This is because each rational curve $L_{i,j} \subset X_{0}$ has self-intersection $(-1)$, and therefore has normal bundle $N_{L_{i,j}/X \times \Delta} \cong \so(-1)\oplus \so.$

On $\F_{1}$, we let $S$ denote the divisor class of a {\sl codirectrix}, a section class having self-intersection $+1$.
	We denote by $R$ the ruling line class.  We let $S_{i,j}$ and $R_{i,j}$ denote the corresponding divisor classes on $F_{i,j}$.  

Let $\scl$ be the line bundle $\beta^{*}(\so_{X \times \Delta}(5H - 2E_{1}-2E_{2}-2E_{3}))$, and let $p'_{i}(t) \subset \scx$ denote the lifts of $p_{i}(t)$ to $\scx$. In other words, $\{p'_{i}(t) \}_{i=1, ... 10}$ are the closures of the points $\{p_{i}(\eta)\} \in X_{\eta} = \scx_{\eta}$ in $\scx$. 

By the generality assumptions on the $1$-parameter family of points $\{p_{i}(t)\}_{i=1, ... , 10}$ in $X \times \Delta$, we may assume the following about the central configuration of points $p'_{i}(0)$ in $\scx_{0}$:
\begin{enumerate}
\item The points $p'_{1}(0), p'_{2}(0)$ are general in $F_{0,1}$.
\item The points $p'_{3}(0), p'_{4}(0)$ are general in $F_{0,2}$.
\item The points $p'_{5}(0), p'_{6}(0)$ are general in $F_{1,2}$.
\item The points $m_{i,j} := M \cap L_{i,j}$ are general in $F_{i,j}$ with respect to the other two points mentioned in each part above. 
\end{enumerate}

\begin{figure}
\begin{equation}
	\nonumber
  \begin{tikzcd} 
    F_{i,j} \ar{d} \arrow[hookrightarrow]{r} & \scx := Bl_{L_{i,j} \times \left\{ 0 \right\}} \left(  X \times \Delta\right) \ar {d}{\beta}  & p'_i(t) \arrow[hookrightarrow]{l} \\
     L_{i,j} \times \left\{ 0 \right\} \arrow[hookrightarrow]{r} & X \times \Delta 
     & p_i(t) \arrow[hookrightarrow]{l} \arrow[-,double equal sign distance]{u} 
  \end{tikzcd}\end{equation}
\caption{A pictorial summary of relevant schemes}
\end{figure}

Set \[\scl ' := \scl(-F_{0,1} - F_{0,2} - F_{1,2}).\]
The following lemma is straightforward to prove.

\begin{lemma}\label{lem:restrictionsofLprime}
The line bundle $\scl '$ restricts to $\so_{F_{i,j}}(S_{i,j} + R_{i,j})$ on the exceptional divisors $F_{i,j} \subset \scx_{0}$ and restricts to $\so_{X}(2H)$ on $X \subset \scx_{0}$.
\end{lemma}

\begin{remark}
For the benefit of the reader, we give an alternate description of the linear system $|S + R|$ on $\F_{1}$ appearing in the above lemma.  If we view $\F_{1}$ as the blow up of $\bp^{2}$ at a point $q \in \bp^{2}$, then the linear system $|S+R|$ is the system of conics through the point $q$.  

In particular, if three more general points are chosen on $\F_{1}$, there will be a unique pencil of curves in $|S+R|$ containing them.
\end{remark}

Now consider the sheaf $\scf := \sci_{\{p'_{i}(t)\}_{i=1, ..., 10}} \otimes \scl '$ . 

\begin{lemma}\label{lem:cohomologybasechange}
The $k[[t]]$-module $H^{0}(\scx, \scf)$ is free of rank $2$.  Furthermore, the restriction map \[H^{0}(\scx, \scf) \to H^{0}(\scx_{0}, \scf|_{\scx_{0}})\] is surjective.
\end{lemma}

\begin{proof}
$\scf$ is a torsion free sheaf, hence $H^{0}(\scx, \scf)$ is a torsion free $k[[t]]$-module, i.e. it is free. \autoref{lem:generalquinticpencil} tells us that the rank must be $2$.

By Grauert's theorem, it suffices to show \[h^{0}(\scx_{0}, \scf|_{\scx_{0}}) = 2.\]

A section $s$ of $\scf|_{\scx_{0}}$ is a section of $\scl ' |_{\scx_{0}}$ vanishing at the ten points $p'_{i}(0)$. We will now analyze what the zero locus of $s$ must be on each of the four components of $\scx_{0}$, beginning with $X$. 

The restriction $s|_{X}$ vanishes on a conic containing $p'_{7}(0), p'_{8}(0), p'_{9}(0),$ and $p'_{10}(0)$. Since the latter three points are collinear lying on the line $M$, such a conic is degenerate, of the form $M \cup N$, where $N$ is any line containing $p'_{7}(0)$. 

The restriction $s|_{F_{0,1}}$ vanishes on a divisor of class $|S_{0,1} + R_{0,1}|$ containing the pair of points $p'_{1}(0), p'_{2}(0)$.  Similar descriptions hold for the remaining two components. 

A section $s$ of $\scf |_{\scx_{0}}$ consists of sections on each component which agree on the intersection curves $L_{i,j}$. 
We claim that such a global section is determined, up to scaling, by its restriction to the component $X$.  Indeed, by choosing a conic of the form $M \cup N$, we determine two points $m_{i,j}, n_{i,j}$ on each line $L_{i,j}$, namely the intersections $M \cap L_{i,j}$, $N \cap L_{i,j}$.  

From the generality assumptions we have imposed, we get that there is a unique curve in the class $|S_{0,1} + R_{0,1}|$ containing the four points  $p'_{1}(0), p'_{2}(0), m_{0,1},$ and $n_{0,1}$.  Similarly for the other components $F_{i,j}$. 
It follows that any global section of $\scf$ is determined, up to scaling, by its restriction to $X$. But the restriction to $X$ is a degenerate conic of the form $M \cup N$ as described above, and therefore $h^{0}(\scx_{0}, \scf|_{\scx_{0}}) = 2$, as we claimed. 
\end{proof}

\begin{lemma}\label{lem:centralbaselocus}
The common zero locus of all sections of $\scf |_{\scx_{0}}$ is $M \cup \{p'_{1}(0) , ... , p'_{6}(0) , p'_{7}(0)\}$.
\end{lemma}

\begin{proof}
This follows from the description of the zero loci of sections of  $\scf |_{\scx_{0}}$ found in the proof of \autoref{lem:cohomologybasechange}.
\end{proof} 

Let $\langle f_{1}, f_{2} \rangle$ be a $k [[t]]$-basis for $H^{0}(\scx, \scf)$.
Note that, $\langle f_{1}, f_{2} \rangle$ restricts to a basis $H^{0}(\scx_{0}, \scf|_{\scx_{0}})$ by \autoref{lem:cohomologybasechange}.

\begin{lemma}\label{lem:familybaseloci}
Maintain the notation above, and let $\scy \subset \scx$ defined by $f_{1} = f_{2} = 0$ be the common zero scheme.  Then, as schemes, $\scy \cap \scx_{0} =  M \cup \{p'_{1}(0) , ... , p'_{6}(0) , p'_{7}(0)\}$ and $\scy \cap \scx_{\eta} = \{p_{1}(\eta), ... , p_{10}(\eta), a(\eta), b(\eta), c(\eta)\}$.
\end{lemma}

\begin{proof}
The generality assumptions on the original family of points $p_{i}(t)$ and \autoref{lem:generalquinticpencil} ensure the statement regarding $\scy \cap \scx_{\eta}$. Then, $\scy \cap \scx_{0} =  M \cup \{p'_{1}(0) , ... , p'_{6}(0) , p'_{7}(0)\}$ follows from \autoref{lem:centralbaselocus}.
\end{proof}

Now let $\{a'(t), b'(t), c'(t)\}$ denote the closures of $\{a(\eta), b(\eta), c(\eta)\}$ in $\scx$.

\begin{corollary}\label{cor:closureofresidual}
The scheme $\{p'_{8}(t), p'_{9}(t), p'_{10}(t), a'(t), b'(t), c'(t)\} \cap \scx_{0}$ is contained in the line $M \subset X \subset \scx_{0}$. 
\end{corollary}

\begin{proof}
This follows from \autoref{lem:familybaseloci}.  Indeed, \[\{p'_{8}(t), p'_{9}(t), p'_{10}(t), a'(t), b'(t), c'(t)\} \cap \scx_{0}\] must be a subscheme of $\scy \cap \scx_{0} =  M \cup \{p'_{1}(0) , ... , p'_{6}(0) , p'_{7}(0)\}$.  The sections $\{a'(t), b'(t), c'(t)\}$ cannot limit to any of the seven isolated points $\{p'_{1}, ... , p'_{7}\}$, since these seven points occur with multiplicity one in the scheme $\scy \cap \scx_{0}$. Therefore, the points $\{a'(0), b'(0), c'(0)\}$  must limit to $M$. 
\end{proof}

\sssec{Proof of \autoref{theorem:DegenerateConfigurationInClosure}.}
\label{sssec:degeneration-proof}

\begin{proof}	
	A one parameter family of thirteen points \[\{p_{1}(t), ... , p_{10}(t), a(t), b(t), c(t)\}\] discussed above limits, at $t=0$, to a configuration which we call $\Gamma_{13} \subset \bp^{2}$.  (Technically, $\Gamma_{13}$ is a set in $X$, but we view it as a set in $\bp^{2}$, since $\Gamma_{13}$ avoids the exceptional divisors in $X$.)

Now we argue that the pair $(\{x_{0},y_{0},z_{0}\}, \Gamma_{13})$ is isolated in its fiber under the projection $\pi_{2} \from \Phi \to \Hilb_{13}\bp^2$.

It suffices to show that there are only finitely many noncollinear triads $T \subset \bp^{2}$ disjoint from $\Gamma_{13}$ for which there is a pencil of quintics $C_{t}$ all singular at $T$ and containing $\Gamma_{13}$.

Any pencil of quintics containing $\Gamma_{13}$ must contain the line $M$ in its base locus, since $6$ of the points of $\Gamma_{13}$, $\{p_{8}(0), p_{9}(0), p_{10}(0), a(0), b(0), c(0)\}$ lie on this line (\autoref{cor:closureofresidual}). Therefore, the residual quartic curves of the pencil, denoted $C'_{t}$, form a pencil of curves singular at $T$, and containing $\{p_{1}(0), ... , p_{6}(0), p_{7}(0)\}$ in its base locus.  Note that the set $\{p_{1}(0), ... , p_{6}(0), p_{7}(0)\}$ is a general set of seven points in the plane.

By degree considerations, a pencil of quartics $C'_{t}$ singular at $T$ and having $7$ remaining points in its base locus is forced to have an entire curve $B$ in its base locus.   The curve $B$ must have degree $1, 2,$ or $3$.  

A straightforward combinatorial check shows that if the three points of $T$ are not collinear, the curve $B$ must be the union of three lines joined by three pairs of points among the set $\{p_{1}(0) ,... ,p_{6}(0), p_{7}(0)\}$, and the triad $T$ is the vertices of the triangle $B$.  All told, there are only finitely many possibilities for $T$, which in turn implies that $(\{x_{0},y_{0},z_{0}\}, \Gamma_{13})$ is isolated in its fiber under projection $\pi_{2} \from \Phi \to \Hilb_{13}\bp^2$. 
\end{proof}

\begin{remark}\label{remark:630}
The method of proof for \autoref{theorem:DegenerateConfigurationInClosure} actually shows that there are at least $630$  $3$-Veronese surfaces through thirteen general points in $\bp^{9}$. The reason for this is that we made several
choices in constructing an isolated point of the incidence correspondence.
We choose one of the points $p_1, \ldots, p_6, p_7$ to not lie
in the triangle containing the nodal base locus, and we then chose a division
of the remaining six points into three pairs of two points.
In total there are $7 \cdot \frac{6!}{2! \cdot 2! \cdot 2!} = 630$
such choices, and hence at least $630$ isolated points.
Then it follows that there are at least $630$ $3$-Veronese surfaces
through a general set of $13$ points by
\cite[II.6.3, Theorem 3]{shafarevich:basic-algebraic-geometry-1}.
\end{remark}

\subsection{The remaining del Pezzo surfaces and tricanonical genus 3 curves} 
\label{ssec:degrees-7-and-8}
\subsubsection{Degree $8$.} 
Weak interpolation for degree $8$, type $1$ del Pezzo surfaces asks
 whether such surfaces pass through $12$ general points $\Gamma_{12} \subset \bp^{8}$,
by \autoref{table:del-Pezzo-conditions}. In fact, weak interpolation
for degree $8$, type $1$ del Pezzo surfaces follows almost immediately from our knowledge of interpolation for degree $9$ del Pezzos. 

\begin{corollary}
  \label{corollary:degree-8-type-1-interpolation}
Degree $8$ del Pezzos isomorphic to the Hirzebruch surface $\F_{1}$ satisfy weak interpolation.
\end{corollary}

\begin{proof}

Indeed, Let $A(\Gamma_{12}) \subset \bp^{2}$ be the associated set. Now append a general thirteenth point $p \in \bp^{2}$ and let $B_{13} \subset \bp^{2}$ be the union. 

As follows from \autoref{prop:6h-3x-3y-3z}, association for $B_{13}$ is induced by the linear system of sextics having triple points at a singular triad for $B_{13}$.  Now we take the subsystem of such sextics with further basepoint at the chosen point $p$.  The resulting subsystem induces association for $A(\Gamma_{12})$, and maps $\bp^{2}$ birationally to a degree $8$ del Pezzo containing $\Gamma_{12}$, abstractly isomorphic to the Hirzebruch surface $\F_{1}$.
\end{proof} 

\begin{remark}
The parameter count suggests that there will be a two dimensional family of del Pezzo $8$'s through a general $\Gamma_{12}$. The argument above also has two dimensions of freedom in the choice of auxiliary point $p$.
\end{remark}

\subsubsection{Degree $7$ del Pezzo surfaces} 

\begin{corollary}
  \label{corollary:degree-7-interpolation}
Degree $7$ del Pezzo surfaces satisfy weak interpolation.
\end{corollary}

\begin{proof}
The parameter count says that weak interpolation for such surfaces is equivalent to asking them to pass through $11$ general points $\Gamma_{11} \subset \bp^{7}$.  We now proceed analogously to the previous case: We now append two general auxiliary points $p,q \in \bp^{2}$ to the associated set $A(\Gamma_{11}) \subset \bp^{2}$. 
\end{proof}

\begin{remark}
Paralleling the degree $8$ case, the dimension of del Pezzo 7's through eleven general points is four dimensional, as is the dimension of the space of auxiliary pairs $p,q \in \bp^{2}$. 
\end{remark}

\begin{remark}
The reason why this method fails for degree $6$ del Pezzos is that the number of points required by weak interpolation is not $10$, as the current pattern would suggest.  Rather, the required number of points is eleven, and therefore we needed a separate argument. 
\end{remark}

\sssec{Genus 3 tricanonical curves}

As a bonus, we show that the closed locus of the Hilbert scheme of
degree $12$ genus $3$ curves in $\bp^9$ which are tricanonically
embedded satisfy interpolation.

\begin{corollary}
	\label{corollary:tricanonical}
	The closed locus of the Hilbert scheme of degree $12$ genus $3$
	curves in $\bp^9$ which are tricanonically embedded satisfy interpolation.
\end{corollary}
\begin{proof}
	First, note that there is a $105 = 99 + 6$ dimensional space of
	tricanonically embedded curve, where $99 = \dim \PGL_{10}$ and
	$6 = \dim \mathscr M_3$. In this case, by \ref{interpolation-sweep},
	we have to show that there is a 1 dimensional
	family of such curves through $13$ points, sweeping out a surface.
	But, since we know a $3$-Veronese surface passes through these $13$
	points, we have a 1 dimensional family of tricanonical genus $3$ curves
	sweeping out this Veronese surface passing through $13$ points, as desired.
\end{proof}

\subsection{Enumerating singular triads: observations and obstacles}
\label{ssec:enumerating-singular-triads}
We now discuss the obstacle we face in the computation of the number of singular triads for a general set $\Gamma_{13} \subset \bp^{2}$.  Set $S = \Bl_{\Gamma_{13}}\bp^{2}$, and let $\scl = \so_{S}(5H - E_{1} - ... -E_{13})$.

We should set up the problem on a compact, smooth space.  A natural choice is the Hilbert scheme $\Hilb_{3}S$ parameterizing length three subschemes of $S$. 

The universal scheme $\scz \subset \Hilb_{3}S \times S$ 
has two obvious projections   $\pi_{1} \from \scz \to \Hilb_{3}S$ and $\pi_{2} \from \scz \to S$.  Next, we consider the sheaf \[\scf = \pi_{1*}(\pi_{2}^{*}\scl/(\sci^{2}_{\scz} \otimes \pi_{2}^{*}\scl)).\]

Unfortunately, the sheaf $\scf$, which has generic rank $9$, fails to be locally free precisely along the locus $F \subset \Hilb_{3}S$ parameterizing degree $3$ schemes of the form $\Spec k[x,y]/(x^{2},xy,y^{2})$, also known as the ``fat points''. 

There is a natural restriction map \[\rho \from \so_{\Hilb_{3}S}^{\oplus 8} \to \scf.\]  If $\scf$ were locally free of rank $9$, we could attempt to use Porteous' formula to find the locus where the rank of $\rho$ drops to $6$. Since $\scf$ is not locally free, this approach fails from the outset.

One fix would be to work on a blow up of $\Hilb_{3}S$ along the locus $F$, but then it's unclear what should replace the sheaf $\scf$.  What's more, we would need to identify the Chern classes of the replacement sheaf in the Chow ring of $\Bl_{F}(\Hilb_{3}S)$, which is challenging in its own right. See \cite{elencwajg:l-anneau-de-chow-des-triangles-du-plan}.

Another potential fix would be to work in the {\sl nested} Hilbert scheme $\Hilb_{2,3}S$ parameterizing, $X_{2} \subset X_{3} \subset S$, pairs of length $2$ subschemes contained in length $3$ subschemes. It is known that $\Hilb_{2,3}S$ is smooth, and it has a generically finite, degree $3$ map to $\Hilb_{3}S$ given by forgetting $X_{2}$, which has one dimensional fibers (isomorphic to $\bp^{1}$) precisely over $F \subset \Hilb_{3}S$. This space $\Hilb_{2,3}S$ might be better suited for replacing the problematic sheaf $\scf$ above.  Finding a solution to these issues is the subject of ongoing work. 
As further references for enumerative geometry in the Hilbert scheme of three points, see \cite{russell:counting-singular-plane-curves-via-hilbert-schemes}, \cite{russell:degenerations-of-monomial-ideals}, and \cite{harrisP:severi-degrees-in-cogenus-3}.

\section{Further Questions} 
\label{sec:questions}

We conclude with some interesting open interpolation problems.

Since weak interpolation is equivalent to interpolation
for del Pezzo surfaces of degrees $3,4,5$, and $9$,
it is immediate from \autoref{theorem:main} that 
del Pezzo surfaces of degree $3, 4, 5$ and $9$ surfaces satisfy interpolation,
while the remaining del Pezzo surfaces satisfy weak interpolation.

\begin{question}
	\label{question:}
	Do all del Pezzo surfaces satisfy strong interpolation?
	If so, how many del Pezzo surfaces meet
	a collection of points and a linear space,
	as given in \autoref{table:del-Pezzo-conditions}?
\end{question}

It was mentioned in the introduction that plane conics
constitute all anticanonically embedded Fano varieties of dimension 1 and
del Pezzo surfaces constitute those of dimension 2.
As we have seen these both satisfy weak interpolation.
Further, there is a complete classification of Fano
varieties in dimension 3 \cite{iskovskikhP:fano-varieties}. 
Unfortunately, it is immediately clear that not all
Fano varieties in dimension more than 3 satisfy interpolation.
A counterexample is provided by the complete intersection of
a quadric and cubic hypersurface in $\bp^5$. 
This leads to the following question:
\begin{question}
	\label{question:}
	Which Fano threefolds, embedded by their
	anticanonical sheaf, satisfy weak interpolation?
	Which Fano threefolds, embedded by their anticanonical sheaf,
	satisfy interpolation? Which Fano varieties in dimension more than $3$ satisfy
	interpolation?
\end{question}

In another direction, we may note that surfaces of minimal degree satisfy
interpolation, that is, surfaces of degree $d-1$ in $\bp^d$ satisfy
interpolation by \autoref{thm:stronginterpscroll}.
In this paper, we show that all smooth surfaces of one more than minimal
degree, which are not projections of surfaces of degree $d$ from $\bp^{d+1}$ 
as described in \cite[Theorem 2.5]{coskun:the-enumerative-geometry-of-del-Pezzo-surfaces-via-degenerations}.
satisfy interpolation. That is, surfaces of degree $d$ in $\bp^d$ satisfy interpolation. 

\begin{question}
	\label{question:}
	Do all smooth surfaces of degree $d$ in $\bp^d$ satisfy
	interpolation? Equivalently, using
	\cite[Theorem 2.5]{coskun:the-enumerative-geometry-of-del-Pezzo-surfaces-via-degenerations}
	\autoref{theorem:main},
			do projections of surfaces of minimal
	degree from a point 
	satisfy interpolation?
\end{question}

While all smooth linearly normal nondegenerate surfaces of degree $d$ in $\bp^d$
satisfy weak interpolation, note that not all surfaces of degree $d+2$
in $\bp^d$ will satisfy interpolation.
This is because the complete
intersection of a quadric and cubic hypersurface in $\bp^4$ does not
satisfy interpolation.
So, in some way, surfaces of degree $d + 1$ in $\bp^d$
are the turning point between surfaces satisfying interpolation
and surfaces not satisfying interpolation.
This leads naturally to the following question.

\begin{question}
	\label{question:}
	Do surfaces of degree $d + 1$ in $\bp^d$ satisfy interpolation?
\end{question}

From \autoref{thm:stronginterpscroll}, we know that varieties of dimension
$k$ and degree $d$ in $\bp^{d+k-1}$ satisfy interpolation.
In this paper we have seen that varieties of degree $2$ and
dimension $2$ (which are nondegenerate and not projections
of varieties of minimal degree)
in
$\bp^{d + 2 - 2} = \bp^d$ satisfy interpolation.
This too  offers an immediate generalization.

\begin{question}
	\label{question:}
	Do varieties of dimension $k$ and degree $d$ in
	$\bp^{d+k-2}$ satisfy interpolation?
\end{question}

Similarly, we have seen the very beginnings of interpolation
for Veronese embeddings. That is, by the discussion in
\ref{sssec:castelnuovos-lemma},
all rational normal curves which are the Veronese embeddings
of $\bp^1$ satisfy interpolation. In general, interpolation
of the $r$-Veronese embedding of $\bp^n$ which is the image of
$\bp^n \rightarrow \bp^{\binom{n+r}{n}-1}$ is equivalent
to the question of whether the Veronese surface
passes through $\binom{n+r}{r} + n +1$ points.
Unlike the del Pezzo surfaces and rational normal scrolls,
Veronese embeddings are a class of varieties for
which interpolation only imposes point conditions, and not
an additional linear space condition. Perhaps
this coincidence may be helpful in finding the solution to
the following question.

\begin{question}
	\label{question:veronese-interpolation}
	Does the image of the $r$-Veronese embedding $\bp^n \rightarrow \bp^{\binom{n+r}{r}-1}$ satisfy interpolation?
	That is, is there a Veronese embedding containing
	$\binom{n+r}{r} + n + 1$ general points in $\bp^{\binom{n+r}{r}-1}$?
\end{question}

If the answer to ~\autoref{question:veronese-interpolation}
is affirmative, it would be very interesting to know how many
Veronese varieties pass through the correct number of points.
Using \ref{sssec:castelnuovos-lemma},
we know there is precisely one $r$-Veronese $\bp^1$ through 
$\binom{r+1}{1} + 1 +1  = r+3$ points in $\bp^r$.
Additionally,  \autoref{theorem:counting-veronese-interpolation}
tells us there are $4$ $2$-Veronese surfaces through $9$ general
points in $\bp^5$. In this paper, we have shown that there are at least $630$ $3$-Veronese surfaces through $13$ general points in $\bp^{9}$. See \autoref{remark:630}.

\begin{question}
	\label{question:veronese-enumeration}
	How many $r$-Veronese varieties of dimension
	$n$ pass through
	$\binom{n+r}{r} + n + 1$ general points in $\bp^{\binom{n+r}{r}-1}$?
\end{question}

We have also seen in Coble's work that any two $2$-Veronese
surfaces through $9$ general points in $\bp^5$
intersect along a genus 1 curve through those $9$ points.
This leads to the question:
\begin{question}
	\label{question:intersect-curve}
	Suppose there are at least two $r$-Veronese varieties
	of dimension $n$ passing through
	$\binom{n+r}{r} + n + 1$ general points in $\bp^{\binom{n+r}{r}-1}$.
	Do they have positive dimensional intersection?
\end{question}

\appendix
\section{Interpolation in General}
\label{sec:interpolation-in-general}

In this section, we define various notions of interpolation, and
prove they are all equivalent under mild hypotheses in
~\autoref{theorem:equivalent-conditions-of-interpolation}.
Many of the results are likely well-known to experts, but we could not find precise references, so we give proofs for completeness.

\ssec{Definition and Equivalent Characterizations of Interpolation}
We now lay out the key definitions of interpolation.
First, we describe a more formal way of expressing interpolation
in ~\autoref{definition:interpolation}.
This comes in two flavors: interpolation, and pointed interpolation.
The latter also keeps track of the points at which the planes
meet the given variety. Then, we give a cohomological
definition in ~\autoref{definition:vector-bundle-interpolation}.

\begin{definition}
	\label{definition:Hilbert-scheme-component}
	Let $X \subset \bp^n$ be projective scheme with a fixed embedding into projective space which lies on a unique irreducible component
	of the Hilbert scheme.
	Define $\hilb X$ to be the irreducible component of the Hilbert scheme on which $[X]$ lies, taken with reduced scheme structure.
	If $\sch$ is the Hilbert scheme of closed subschemes of $\bp^n$ over $\spec k$ and $\scv$
	is the universal family over $\sch$, then define
	define $\uhilb X$ to be the universal family over $\hilb X$, defined as the
	fiber product
	\begin{equation}
	\nonumber
		\nonumber
		\begin{tikzcd} 
		  \uhilb X \ar{r} \ar{d} & \scv \ar{d} \\
		  \hilb X \ar{r} & \sch.
		\end{tikzcd}\end{equation}
\end{definition}

\begin{definition}
	\label{definition:lambda-constraints}
	Given an integral subscheme of the Hilbert scheme $U$
	parameterizing subschemes of $\bp^n$ of dimension $k$,
	call a sequence
	\begin{align*}
	\lambda := \left( \lambda_1, \ldots, \lambda_m \right)
\end{align*}
{\bf admissible} if it satisfies the following conditions:
\begin{enumerate}
	\item $\lambda$ is a weakly decreasing sequence.
		That is, $\lambda_1 \geq \lambda_2 \geq \cdots \geq \lambda_m$,
	\item for all $1 \leq i \leq m$, we have $0 \leq \lambda_i \leq n - k$,
	\item and
\begin{align*}
	\sum_{i=1}^m \lambda_i \leq \dim U.
\end{align*}
\end{enumerate}
\end{definition}

\begin{definition}
	\label{definition:interpolation}	
	Let $U$ be an integral subscheme of the Hilbert scheme
	parameterizing subschemes of $\bp^n$ of dimension $k$
	and let $\scv(U)$ denote the universal family over $U$.
	Let $\lambda$ be admissible and let 
	$\Lambda_i$ be a plane of dimension $n - k - \lambda_i$ for $1 \leq i \leq m$.
Define
\begin{align*}
	\Psi := \left( \uhilb {\Lambda_1} \times_{\bp^n} \scv(U) \right) \times_{U} \cdots \times_{U} \left( \uhilb {\Lambda_m} \times_{\bp^n} \scv(U) \right).
\end{align*}
Then, since $\hilb {\Lambda_i} \cong Gr(n- k - \lambda_i + 1, n+1)$, define $\Phi$ to be the scheme theoretic image of the composition
\begin{equation}
	\nonumber
	\begin{tikzcd} 
	  \Psi \ar{r} & U \times \prod_{i=1}^m Gr(n - k - \lambda_i + 1, n + 1) \times (\bp^n)^m \ar{d} \\
		 & U \times \prod_{i=1}^m Gr(n - k - \lambda_i + 1, n + 1).
	\end{tikzcd}\end{equation}
We have natural projections
\begin{equation}
	\nonumber
	\begin{tikzcd}
		\qquad & \Phi \ar {ld}{\pi_1} \ar {rd}{\pi_2} & \\
		U  &&  \prod_{i=1}^m Gr(n - k - \lambda_i +1, n+1)
	 \end{tikzcd}\end{equation}
and
\begin{equation}
	\nonumber
	\begin{tikzcd}
		\qquad & \Psi \ar {ld}{\eta_1} \ar {rd}{\eta_2} & \\
		U  &&  \prod_{i=1}^m Gr(n - k - \lambda_i +1, n+1).
	 \end{tikzcd}\end{equation}

Define $q$ and $r$ so that $\dim U = q \cdot(n-k) + r$ with $0 \leq r < n - k$. Then,
$U$ satisfies
\begin{enumerate}
	\item {\bf $\lambda$-interpolation}
if the projection map $\pi_2$ is surjective.
	\item {\bf weak interpolation} if $U$ satisfies $((n - k)^q)$-interpolation
	\item {\bf interpolation} if $U$ satisfies $((n - k)^q, r)$-interpolation
	\item {\bf strong interpolation}
if $U$ satisfies $\lambda$-interpolation for all admissible $\lambda$.
\end{enumerate}

We define  $\lambda$-pointed interpolation, weak pointed interpolation, pointed interpolation, strong pointed interpolation similarly. More precisely, we say that $U$ satisfies
\begin{enumerate}
	\item {\bf $\lambda$-pointed interpolation} if $\eta_2$
is surjective
\item {\bf weak pointed interpolation} if $U$ satisfies
	$\left( (n - k)^q \right)$-pointed interpolation
\item {\bf pointed interpolation} if $U$ satisfies
$\left( (n - k)^q, r \right)$-pointed interpolation
\item {\bf strong pointed interpolation} if $U$ satisfies
$\lambda$-pointed interpolation for all admissible $\lambda$.
\end{enumerate}

If $X \subset \bp^n$ lies on a unique irreducible
component of the Hilbert scheme $\hilb X$, we say $X$ satisfies $\lambda$-interpolation (and all variants as
above) if $\hilb X$ satisfies $\lambda$-interpolation.
\end{definition}
\begin{remark}
	\label{remark:pointed-interpolation-equivalent-to-non-pointed}
	Note that $U$ satisfies $\lambda$-interpolation if and only if it satisfies $\lambda$-pointed interpolation:
	$\eta_2$ factors through $\Phi$,
	and the restriction map $\Psi \rightarrow \Phi$ is surjective,
	so $\eta_2$ is surjective if and only if $\pi_2$ is.
	Nevertheless, it is useful to refer to these two notions separately, which is why we give
	them two separate names.
\end{remark}

\begin{definition}[Interpolation of locally free sheaves, see Definition 3.1 of ~\cite{atanasov:interpolation-and-vector-bundles-on-curves}]
	\label{definition:vector-bundle-interpolation}
	Let $\lambda$ be admissible and let $E$ be a locally free sheaf on a scheme $X$
	with $H^1(X, E) = 0$.
	Choose points $p_1, \ldots, p_m$ on $X$ and
	vector subspaces $V_i \subset E|_{p_i}$ for $1 \leq i \leq m$
	with $\codim V_i = \lambda_i$.
	Then, define $E'$ so that we have an exact sequence of coherent sheaves on $X$
	\begin{equation}
		\label{equation:sequence-cohomological-interpolation-definition}
		\begin{tikzcd}
		  0 \ar{r} & E' \ar{r} & E \ar{r} & \oplus_{i=1}^m E|_{p_i}/V_i \ar{r} & 0.
		\end{tikzcd}\end{equation}
	We say $E$ satisfies {\bf $\lambda$-interpolation}
	if there exist points $p_1, \ldots, p_n$ as above
	and subspaces $V_i \subset E|_{p_i}$ as above
	so that
	\begin{align*}
		h^0(E) - h^0(E') = \sum_{i=1}^m \lambda_i.
	\end{align*}

	Write $h^0(E) = q \cdot \rk E + r$ with $0 \leq r < \rk E$.
	We say $E$ satisfies
	\begin{enumerate}
		\item {\bf weak interpolation} if it satisfies
			$\left( (\rk E)^q \right)$
			interpolation.
		\item {\bf interpolation} if it satisfies
			$\left( (\rk E)^q, r \right)$ interpolation.
		\item {\bf strong interpolation} if it satisfies
			$\lambda$-interpolation for all
			admissible $\lambda$.
	\end{enumerate}
\end{definition}
\begin{remark}
	\label{remark:}
	See ~\cite[Section 4]{atanasovLY:interpolation-for-normal-bundles-of-general-curves} for further useful properties of interpolation.
	While some of the discussion there is specific
	to curves, much of it generalizes immediately to higher
	dimensional varieties. 
\end{remark}

We now come to the main result of the section. 
Because it has so many moving parts, after stating
it, we postpone its proof until
~\autoref{ssec:proof-of-equivalent-conditions-of-interpolation},
after we have developed the tools necessary to prove it.

Perhaps the most nontrivial consequence of ~\autoref{theorem:equivalent-conditions-of-interpolation}
is that it implies the equivalence of interpolation and strong interpolation
for $\hilb X$ when $X$ is a smooth projective scheme with $H^1(X, N_{X/\bp^n}) = 0$,
over an algebraically
closed field of characteristic $0$.

\begin{theorem}
	\label{theorem:equivalent-conditions-of-interpolation}
	Assume $X \subset \bp^n$ is an integral projective scheme
	lying on a unique irreducible component of the Hilbert scheme.
	Write $\dim \hilb X = q \cdot \codim X + r$ with $0 \leq r < \codim X$.
	The following are equivalent:
	\begin{enumerate}
		\item[\customlabel{interpolation-definition}{(1)}]$\hilb X$ satisfies interpolation.
		\item[\customlabel{interpolation-pointed}{(2)}] $\hilb X$ satisfies pointed interpolation.
		\item[\customlabel{interpolation-dominant}{(3)}] The map $\pi_2$ given in 
			~\autoref{definition:interpolation} 
			for $\lambda = (\left(\codim X \right)^q, r)$
			is dominant.
		\item[\customlabel{interpolation-finite}{(4)}] The map $\pi_2$ given in 
			~\autoref{definition:interpolation} 
			for $\lambda = (\left(\codim X \right)^q, r)$
			is
			generically finite.
		\item[\customlabel{interpolation-isolated}{(5)}] The scheme $\Phi$ defined in 
			~\autoref{definition:interpolation}
			for $\lambda = (\left(\codim X \right)^q, r)$
			has a closed point $x$ which is isolated in
			its fiber $\pi_2^{-1}(\pi_2(x))$.
	\item[\customlabel{interpolation-pointed-dominant}{(6)}] The map $\eta_2$ given in 
			~\autoref{definition:interpolation} 
			for $\lambda = (\left(\codim X \right)^q, r)$
			is dominant.
		\item[\customlabel{interpolation-pointed-finite}{(7)}] The map $\eta_2$ given in 
			~\autoref{definition:interpolation} 
			for $\lambda = (\left(\codim X \right)^q, r)$
			is
			generically finite.
		\item[\customlabel{interpolation-pointed-isolated}{(8)}] The scheme $\Psi$ defined in 
			~\autoref{definition:interpolation}
			for $\lambda = (\left(\codim X \right)^q, r)$
			has a closed point $x$ which is isolated in
			its fiber $\eta_2^{-1}(\eta_2(x))$.
		\item[\customlabel{interpolation-naive}{(9)}] For any set of $q$ points in $\bp^n$ and an $(\codim X - r)$-dimensional
			plane $\Lambda \subset \bp^n$, there exists
			an element $[Y] \in \hilb X$ so that $Y$ contains those points and
			meets $\Lambda$.
		\item[\customlabel{interpolation-sweep}{(10)}] For any set of $q$ points in $\bp^n$,
			the subscheme of $\bp^n$ swept out by varieties
			of $\hilb X$ containing those points is
			$\dim X + r$ dimensional.
	\end{enumerate}
	Secondly, the following statements are equivalent:
	\begin{enumerate}[(i)]
		\item[\customlabel{strong-definition}{(i)}] $\hilb X$ satisfies strong interpolation.
		\item[\customlabel{strong-equality}{(ii)}] $\hilb X$ satisfies $\lambda$-interpolation for all $\lambda$ with $\sum_{i=1}^m \lambda_i = \dim \hilb X$.

		\item[\customlabel{strong-pointed}{(iii)}] $\hilb X$ satisfies strong pointed interpolation.
		\item[\customlabel{strong-pointed-equality}{(iv)}] $\hilb X$ satisfies $\lambda$-pointed interpolation for all $\lambda$ with $\sum_{i=1}^m \lambda_i = \dim \hilb X$.
		\item[\customlabel{strong-naive}{(v)}] For any collection of planes $\Lambda_1, \ldots, \Lambda_m$ with $(\dim \Lambda_1, \ldots, \dim \Lambda_n)$ admissible,
			there is some $[Y] \in \hilb X$ meeting
			all of $\Lambda_1, \ldots, \Lambda_m$.
		\item[\customlabel{strong-naive-equality}{(vi)}] For any collection of planes $\Lambda_1, \ldots, \Lambda_m$ with $(\dim \Lambda_1, \ldots, \dim \Lambda_n)$ admissible, 
			with $\sum_{i=1}^m \lambda_i = \dim \hilb X$,
			there is some $[Y] \in \hilb X$ meeting
			all of $\Lambda_1, \ldots, \Lambda_m$.
	\end{enumerate}
Also, \ref{strong-definition}-\ref{strong-naive-equality}
imply
\ref{interpolation-definition}-\ref{interpolation-sweep}.
Thirdly, 
further assume $H^1(X, N_X) = 0$ and $X$ is a local complete intersection.
Then, the following properties are equivalent:
	\begin{enumerate}[(a)]
		\item[\customlabel{cohomological-definition}{(a)}] The sheaf $N_{X/\bp^n}$ satisfies interpolation.
		\item[\customlabel{cohomological-restatement}{(b)}] There is a subsheaf $E' \rightarrow N_{X/\bp^n}$ 
whose cokernel is supported at $q+1$ points if $r > 0$ and $q$ points if $r = 0$,
so that the scheme theoretic support at $q$ of these points has
dimension equal to $\rk N_{X/\bp^n}$
and $H^0(X, E') = H^1(X, E') = 0$.
		\item[\customlabel{cohomological-strong}{(c)}] The sheaf $N_{X/\bp^n}$ satisfies strong interpolation.
		\item[\customlabel{cohomological-sections}{(d)}] For every $d \geq 1$, there exist points
			$p_1, \ldots, p_d \in X$ so that
			\begin{align*}
				\dim H^0(X, N_{X/\bp^n} \otimes \sci_{p_1, \ldots, p_d}) = \max\left\{ 0, h^0(X, N_{X/\bp^n}) - dn \right\}
			\end{align*}
			(cf. ~\cite[Definition 4.1]{atanasovLY:interpolation-for-normal-bundles-of-general-curves}).
			
		\item[\customlabel{cohomological-vanish}{(e)}] For every $d \geq 1$, a general collection of points
			$p_1, \ldots, p_d$ in $X$ satisfies either
			\begin{align*}
				h^0(X, N_{X/\bp^n} \otimes \sci_{p_1, \ldots, p_d}) = 0 && \text{ or } && h^1(X, N_{X/\bp^n}\otimes \sci_{p_1, \ldots, p_d}) = 0
			\end{align*}
			(cf. ~\cite[Proposition 4.5]{atanasovLY:interpolation-for-normal-bundles-of-general-curves}).
		\item[\customlabel{cohomological-boundary}{(f)}] 			A general set of $q$ points
			$p_1, \ldots, p_q$
			satisfy $h^1(X, N_{X/\bp^n} \otimes \sci_{p_1, \ldots, p_q}) = 0$
			and a general set of $q+1$ points
			$q_1, \ldots, q_{q+1}$ satisfy
			$h^0(X, N_{X/\bp^n} \otimes \sci_{p_1, \ldots, p_q}) = 0$
			(cf. ~\cite[Proposition 4.6]{atanasovLY:interpolation-for-normal-bundles-of-general-curves}).
	\end{enumerate}
Additionally, 
retaining the assumptions that $H^1(X, N_X) = 0$ 
	and $X$ is a local complete intersection, 
	and further assuming $X$ is generically smooth,	
	the equivalent conditions
\ref{cohomological-definition}-\ref{cohomological-boundary} imply the equivalent conditions
\ref{interpolation-definition}-\ref{interpolation-sweep} and the equivalent conditions \ref{cohomological-definition}-\ref{cohomological-boundary} imply the equivalent conditions \ref{strong-definition}-\ref{strong-naive-equality}.

Finally, still 
retaining the assumptions that $H^1(X, N_X) = 0$ 
	and that $X$ is a local complete intersection, in the case that $k$ has characteristic $0$,
all statements \ref{interpolation-definition}-\ref{interpolation-sweep}, \ref{strong-definition}-\ref{strong-naive-equality}, \ref{cohomological-definition}-\ref{cohomological-boundary} are equivalent.
\end{theorem}

We develop the tools to prove
~\autoref{theorem:equivalent-conditions-of-interpolation}
in subsections \ref{ssec:tools-irreducibility}, \ref{ssec:tools-dimensions}, and \ref{ssec:tools-additional},
and then give a proof
of ~\autoref{theorem:equivalent-conditions-of-interpolation}
in ~\autoref{ssec:proof-of-equivalent-conditions-of-interpolation}.

\begin{remark}
	\label{remark:smoothness-of-Hilbert-scheme-for-vanishing-normal-bundle-cohomology}
	Note that if $H^1(X, N_{X/\bp^n})= 0$ 
	and $X$ is a local complete intersection,
	(the latter condition is satisfied for all smooth $X$,)
	then
	$X$ has no local obstructions to deformation
	by \cite[Corollary 9.3]{Hartshorne:deformation}.
	So, by ~\cite[Corollary 6.3]{Hartshorne:deformation},
	$[X]$ is a smooth point of the Hilbert scheme.
\end{remark}

\begin{remark}
	\label{remark:}
	We note that the equivalence of all conditions from
	\autoref{theorem:equivalent-conditions-of-interpolation}
	requires the characteristic $0$ hypothesis, as it does not hold in characteristic $2$.

	The $2$-Veronese surface
	over an algebraically closed field of characteristic $2$ provides
	an example of a variety which satisfies interpolation but whose normal
	bundle does not satisfy interpolation, as is shown in \cite[Corollary 7.2.9]{landesman:undergraduate-thesis}. 
\end{remark}

\ssec{Tools for Irreducibility of Incidence Correspondences}
\label{ssec:tools-irreducibility}

A key ingredient for establishing the equivalence of conditions
\ref{interpolation-definition}-\ref{interpolation-sweep}
is the irreducibility of the incidence correspondences $\Phi, \Psi$ 
(~\autoref{definition:interpolation}).
We use this to establish that the following properties of the map
$\pi_2$ (~\autoref{definition:interpolation}) are equivalent:
(1) it is surjective, (2) it is dominant, (3) it is generically finite, 
and (4) it has an isolated point
in some fiber.
Our goal for this subsection is to prove
~\autoref{proposition:irreducible-and-dimension}.

We start with a general upper semicontinuity result, which we will use in 
~\autoref{proposition:irreducible-and-dimension}
to show that if $X$ is integral than so is $\hilb X$, which we will then use in~\autoref{proposition:irreducible-and-dimension}
to conclude that $\Phi$ and $\Psi$ are irreducible of the same dimension.
\begin{proposition}
	\label{proposition:upper-semicontinuous-number-of-components}
	Let $f: X \ra Y$ be a flat proper map of finite type schemes over an arbitrary field
	so that the fibers over the closed points of $Y$ are geometrically reduced. Then, the number of irreducible components of the geometric fiber of a point in $Y$
	is upper semicontinuous on $Y$.
\end{proposition}
This proof is that outlined in nfdc23's comments in ~\cite{MO:why-is-the-number-of-irreducible-components-upper-semicontinuous-in-nice-situations}.
\begin{proof}
	To start, note that by \cite[Th\'eor\`eme 12.2.4(v)]{EGAIV.3}, the set of points in $Y$ with geometrically reduced
fiber is open, and hence all fibers of $f$ are geometrically reduced, as all closed fibers are.
Then, by \cite[Th\'eor\`eme 12.2.4(ix)]{EGAIV.3}, since the geometric fibers of $f$ 
are reduced and hence have no embedded points, we obtain that the total multiplicity,
as defined in \cite[D\'efinition 4.7.4]{EGAIV.2}, is upper semicontinuous.
Since the total multiplicity of a reduced scheme over an algebraically closed field is equal to the number of irreducible
components,
the number of irreducible components of the geometric fibers is upper semicontinuous on the target.
\end{proof}

\begin{proposition}
	\label{proposition:irreducible-and-dimension}
	Suppose $X$ is an integral scheme. Then, $\Phi, \Psi$
	as defined in ~\autoref{definition:interpolation} are irreducible and $\dim \Phi = \dim \Psi$.
\end{proposition}
\begin{proof}
We start by verifying that a general member of $\hilb X$ is integral if $X$ is.
The map $\uhilb X \rightarrow \hilb X$ has general member which is reduced by
\cite[Th\'eor\`eme 12.2.4(v)]{EGAIV.3}.
Therefore, applying 
~\autoref{proposition:upper-semicontinuous-number-of-components},
the general point of $\hilb X$ has preimage in $\uhilb X$ which
is integral.

We now complete the proof in the case that $m = 1$ as
the general case is completely analogous. We write
$\lambda := \lambda_1, \Lambda := \Lambda_1, p := p_1$ for notational
convenience.
Observe that we have a commutative diagram of natural projections
\begin{equation}
	\nonumber
	\begin{tikzcd} 
		\qquad & \Psi \ar {ld}{\rho_1} \ar {rd}{\rho_2} & \\
		\uhilb X \ar {rd}{\rho_3} & & \Phi \ar{ld}{\rho_4} \\
		\qquad & \hilb X &
	\end{tikzcd}\end{equation}
Observe that since the map $\rho_2$ is surjective, once
we know $\Psi$ is irreducible, $\Phi$ will be too.

Note that the map $\rho_3$ is flat. 
The assumption that the general member of $\hilb X$ is irreducible
precisely says that the general fiber of $\rho_3$ is irreducible.
If we have a flat map to an irreducible base, so that
the general fiber is irreducible, then the source is irreducible
(see, for example, \cite[Lemma 3.2.1]{landesman:undergraduate-thesis}).
Hence, $\uhilb X$ is irreducible.

If we knew $\rho_1$ were a Grassmannian bundle, we would then
obtain that $\Psi$ is also irreducible. 
To see this, we have a fiber square
\begin{equation}
	\nonumber
	\begin{tikzcd} 
	  \Psi \ar{r} \ar{d} &  \left\{ \left( \Lambda, p \right) \subset Gr(\lambda, n+1) \times \bp^n: p \in \Lambda \right\}\ar{d} \\
	  \uhilb X \ar{r} & \bp^n.
	\end{tikzcd}\end{equation}
The left vertical map is a Grassmannian bundle because
the right vertical map is a Grassmannian bundle, so $\Psi$ and $\Phi$
are irreducible.

To conclude, we check that
$\dim \Psi = \dim \Phi$.
Note that if we take the point $(Y, \Lambda)$ in $\Phi$ chosen
so that $\Lambda$ meets $Y$ at finitely many points, the fiber
of $\rho_2$ over that point is necessarily $0$ dimensional.
By upper semicontinuity of fiber dimension for proper maps, there is an open
set of $\Phi$ on which the fiber is $0$ dimensional.
Hence the map is generically finite, so $\dim \Phi = \dim \Psi$.
\end{proof}

\ssec{Tools for Showing Equality of Dimensions of the Source and Target}
\label{ssec:tools-dimensions}

In this subsection we develop some more technical tools for proving
~\autoref{theorem:equivalent-conditions-of-interpolation}.
Our goal for this subsection is to prove ~\autoref{lemma:interpolation-dimension}.
Before embarking on this task, we start with a simple tool for proving the equivalence of
\ref{interpolation-dominant} and \ref{interpolation-isolated}.

\begin{lemma}
	\label{lemma:isolated-fiber-implies-dominant}
	Let $\pi: X \ra Y$ be a proper morphism of locally Noetherian schemes of the same pure dimension. If there is some point $x \in X$ which is isolated in its fiber, then $\dim \im \pi = \dim Y$.
\end{lemma}
\begin{proof}
By Zariski's Main Theorem in Grothendieck's form 
	\cite[Theorem 29.6.1(a)]{vakil:foundations-of-algebraic-geometry}
	there is a nonempty open subscheme $X_0 \subset X$ so that all closed point
	of $X_0$ are isolated in their fibers.
	Therefore, the map restricted to this open subset
	is generically finite, and so
	its image has the same dimension as $Y$.
\end{proof}
\begin{lemma}
	\label{lemma:interpolation-dimension}
	With notation as in ~\autoref{definition:interpolation}, 
	if $\sum_{i=1}^m \lambda_i = \dim \hilb X$, we have $\dim \Phi = \dim \prod_{i=1}^m Gr(\codim X - \lambda_i +1, n+1)$.
	In particular, the source and target of the map $\pi_2$ have the same dimension.
\end{lemma}
\begin{proof}
	This is purely a dimension counting argument. With notation as in~\autoref{definition:interpolation}, $\Phi$ is a fiber product
	of incidence correspondences, $\Phi_i$, where $\Phi_i$ is the scheme theoretic image of
	\begin{align*}
		\Psi_i := \uhilb {\Lambda_i} \times_{\mathbb P^n} \scv(U)
	\end{align*}
	under projection map $\Psi_i \ra \hilb X \times Gr(\codim X - \lambda_i + 1, n+1).$
	We have natural projections
	\begin{equation}
		\nonumber
		\begin{tikzcd}
			\qquad & \Phi_i \ar {ld}{\pi^i_1} \ar {rd}{\pi^i_2} & \\
			 \hilb X && Gr(\codim X -\lambda_i + 1, n+1)
		 \end{tikzcd}\end{equation}
	with
	\begin{align*}
		\Phi = \Phi_1 \times_{\hilb X} \Phi_2 \times_{\hilb X} \cdots \times_{\hilb X} \Phi_m.
	\end{align*}
	The dimension of any fiber of $\pi_1^i$ is $\dim X + \dim Gr(\codim X-\lambda_i, n)$, and hence
	\begin{align*}
		\dim \Phi &= \dim \hilb X + \sum_{i=1}^m (\dim X + \dim Gr(\codim X-\lambda_i, n)) \\
		&= \dim \hilb X + \sum_{i=1}^m (\dim X + (\dim Gr(\codim X-\lambda_i + 1, n + 1) - n-\codim X - \lambda_i)) \\
		&= \dim \hilb X - \sum_{i=1}^m \lambda_i + \sum_{i=1}^m \dim Gr(\codim X-\lambda_i + 1, n + 1) \\
		&= \sum_{i=1}^m \dim Gr(\codim X-\lambda_i + 1, n + 1).\qedhere 
	\end{align*}
\end{proof}

\ssec{Deformation Theory Tools}
\label{ssec:tools-additional}

In this subsection, we prove a result from
deformation theory crucial to proving
the equivalence of the distinct groups of conditions in
~\autoref{theorem:equivalent-conditions-of-interpolation}.

The following proposition is important for establishing
the equivalence between interpolation of
a locally free sheaf and interpolation of a Hilbert scheme; although it might be obvious for experts, we could not find a reference, so we include it for completeness.

\begin{proposition}
	\label{proposition:tangent-space-to-psi}
	Let $\Psi, \eta_2$ and $[X] \in U$ be as in ~\autoref{definition:interpolation}
	and let 
	\begin{align*}
	  \overline p := (X, \Lambda_1, \ldots, \Lambda_m, p_1, \ldots, p_m) \in \Psi
	\end{align*}
	be a closed point of $\Psi$,
	so that $\Lambda_i$ meets $X$ quasi-transversely
	and so that the $p_i$ are distinct smooth points of $X$. Choose subspaces $V_i \subset N_{X/\bp^n}|_{p_i}$ where
$V_i$ is the image of the composition
\begin{equation}
	\nonumber
	\begin{tikzcd}
	  N_{p_i/\Lambda_i} \ar{r} & N_{p_i/\bp^n} \ar{r} & N_{X/\bp^n}|_{p_i}. 
	\end{tikzcd}\end{equation}
For any closed point $\overline q$  of $\Psi$, let \[d\eta_2|_{\overline{q}}: T_{\overline{q}} \Psi \ra T_{\eta_2(\overline q)}\prod_{i=1}^m Gr(\codim X - \lambda_i +1, n+1)\] be the
induced map on tangent spaces. Then, $d{\eta_2}|_{\overline p}$ is surjective 
if and only if the map
\begin{equation}
\nonumber
	\begin{tikzcd}
		H^0(X, N_{X/\bp^n}) \ar{r}{\tau} & H^0(X, \oplus_{i=1}^m N_{X/\bp^n}|_{p_i}/ V_i)
	\end{tikzcd}\end{equation}
is surjective.
\end{proposition}
\begin{proof}
To set things up properly, we will need some definitions.
Recall that $\uhilb {\Lambda_i}$ is the universal family over the Hilbert scheme of $\dim \Lambda_i$ planes in $\bp^n$.
That is, it is the universal family over $Gr(\codim X - \lambda_i + 1, n+1)$.
Next, take $\sch$ to be the Hilbert
scheme with Hilbert polynomial equal to that of $X$ and let $\scv$ be the universal family over $\sch$.
Next, define the scheme
\begin{align*}
	\scf &:= (\sch \times_{\bp^n} \uhilb {\Lambda_1}) \times_\scv \cdots \times_\scv (\sch \times_{\bp^n} \uhilb {\Lambda_m}) \\
	& \cong (\scu \times_\sch \cdots \times_\sch \scu) \times_{(\bp^n)^m} (\uhilb {\Lambda_1} \times \cdots \times \uhilb {\Lambda_m}),
\end{align*}
where there are $m$ copies of $\scu$ in the first parenthesized expression on the second line.

Note that here $\scf$ is not necessarily the same as $\Psi$ because we need not have $\sch = \hilb X$:
The former is the connected component of the Hilbert scheme containing $X$ while $\hilb X$ is the
irreducible component of the Hilbert scheme containing $X$.
However, we will later explain why the tangent spaces of these two schemes are identical, which
is enough for our purposes.

Now, under our assumption that $p_1, \ldots, p_n$ are distinct, we have a diagram

\begin{equation}
\label{equation:three-by-three-deformation-theory}
\hspace*{-1cm}
\includegraphics[scale=.9]{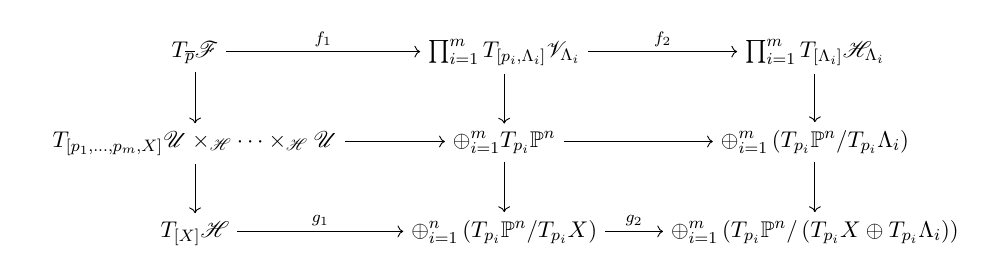}
\end{equation}
\noindent	
in which every square is a fiber square.

First, let us justify why the four small squares of \eqref{equation:three-by-three-deformation-theory} are fiber squares.
The lower right hand square of \eqref{equation:three-by-three-deformation-theory} is a fiber square by elementary linear algebra and the assumption that
$\Lambda_i$ meet $X$ quasi-transversely.
The upper right square of \eqref{equation:three-by-three-deformation-theory} is a fiber square for each $i$
by \cite[Remark 4.5.4(ii)]{sernesi:deformations-of-algebraic-schemes},
as the universal family over the Hilbert scheme is precisely the Hilbert flag scheme of points inside that Hilbert scheme.
Next, the lower left hand square of \eqref{equation:three-by-three-deformation-theory} is a fiber
square because when the points $p_1, \ldots, p_n$ are distinct, the tangent space to this
$n$-fold fiber product of universal families over the Hilbert scheme is the same as the tangent space
to the Hilbert flag scheme of degree $n$ schemes inside schemes with the same Hilbert polynomials as $X$.
Then, the fiber square follows from \cite[Remark 4.5.4(ii)]{sernesi:deformations-of-algebraic-schemes} for this flag Hilbert scheme.
Finally, the upper left square of \eqref{equation:three-by-three-deformation-theory} is a fiber square
because $\scf$ is defined as a fiber product of
$(\scu \times_\sch \cdots \times_\sch \scu)$ and $(\uhilb {\Lambda_1} \times \cdots \times \uhilb {\Lambda_m})$,
and the fiber product of the tangent spaces is the tangent space of the fiber product.

Now, observe that the composition $f_2 \circ f_1$ is precisely the map on tangent spaces $d\eta_2|_{\overline{p}}$.
To make this identification, we need to know that we can naturally identify
$T_{\overline p} \scf \cong T_{\overline p}\Psi$.
However, the assumptions that $H^1(X, N_{X/\bp^n}) = 0$
and that $X$ is lci imply that $[X]$ is a smooth point of the Hilbert scheme.
Because the fiber over $[X]$ of the projection $\Psi \ra \hilb X$ is smooth, it follows that $\Psi$ is smooth at $\overline p$.
For the same reason, it follows that $\scf$ is smooth at the corresponding point $\overline p$.
Therefore, both $\scf$ and $\Psi$ are smooth on some open neighborhood $U$ containing
$\overline p$.
Now, since both $\Psi$ and $\scf$ are defined in terms of fiber products, which agree on some open neighborhood $V$
contained in $U$, it follows that on $V$ we have an isomorphism $\scf|_V \cong \Psi|_V$, and in particular their
tangent spaces are isomorphic. So, we can identify $f_2 \circ f_1$ with $d\eta_2|_{\overline{p}}$.

Since all four subsquares of \eqref{equation:three-by-three-deformation-theory}
are fiber squares, the full square \eqref{equation:three-by-three-deformation-theory} is a fiber square,
and hence $f_2 \circ f_1$ is an isomorphism if and only if $g_2 \circ g_1$ is an isomorphism.

To complete the proof, we only need identify the map $g_2 \circ g_1$ with $\tau$.
But this follows from the identifications
\begin{align*}
	T_{[X]} \sch &\cong H^0(X, N_{X/\bp^n}) \\
	T_{\bp_i} \bp^n/T_{p_i} X &\cong H^0(X, N_{X/\bp^n}|_{p_i}) \\
	T_{\bp_i} \bp^n/\left(T_{p_i} X \oplus T_{p_i} Y  \right) &\cong H^0(X, N_{X/\bp^n}|_{p_i}/V_i).
\end{align*}
The first isomorphism follows from \cite[Theorem 1.1(b)]{Hartshorne:deformation}.
The second isomorphism holds because the normal exact sequence
\begin{equation}
\nonumber
	\begin{tikzcd}
	  0 \ar{r} &  T_{p_i}X \ar{r} & T_{p_i} \bp^n \ar{r} & N_{X/\bp^n}|_{p_i} \ar{r} & 0 
	\end{tikzcd}\end{equation}
is exact on global sections, as all sheaves are supported at $p_i$.
The third isomorphism holds because
$\left( T_{p_i} \bp^n / \left( T_{p_i} X \oplus T_{p_i} \Lambda_i \right) \right)$ can be viewed as the
quotient of $T_{p_i} \bp^n$ first by $T_{p_i} X$ and then by the image of $T_{p_i} \Lambda_i$ in that quotient.
However, $T_{p_i} \bp^n / T_{p_i} X \cong N_{X/\bp^n}|_{p_i}$, and then
$V_i$ is by definition the image of $T_{p_i} \Lambda_i$ in $N_{X/\bp^n}|_{p_i}$.
\end{proof}

\ssec{Proof of ~\autoref{theorem:equivalent-conditions-of-interpolation}}
\label{ssec:proof-of-equivalent-conditions-of-interpolation}

\begin{proof}[Proof of ~\autoref{theorem:equivalent-conditions-of-interpolation}]

The structure of proof is as follows:
\begin{enumerate}
	\item Show equivalence of conditions \ref{interpolation-definition}-\ref{interpolation-sweep}
	\item Show equivalence of conditions \ref{strong-definition}-\ref{strong-naive-equality}
	\item Show equivalence of conditions \ref{cohomological-definition}-\ref{cohomological-boundary}
	\item Demonstrate the implications that \ref{cohomological-definition}-\ref{cohomological-boundary} imply \ref{interpolation-definition}-\ref{interpolation-sweep},  \ref{cohomological-definition}-\ref{cohomological-boundary} imply
\ref{strong-definition}-\ref{strong-naive-equality}, and
\ref{strong-definition}-\ref{strong-naive-equality}
imply
\ref{interpolation-definition}-\ref{interpolation-sweep}, in all characteristics. Further, all statements are equivalent in characteristic $0$.
\end{enumerate}
\sssec{Equivalence of Conditions \ref{interpolation-definition}-\ref{interpolation-sweep}}

First, \ref{interpolation-definition} and
\ref{interpolation-pointed} are equivalent as mentioned in~\autoref{remark:pointed-interpolation-equivalent-to-non-pointed},
applied to the case $\lambda = ((\codim X)^q, r)$.

Next, note that a proper map of irreducible schemes of the same dimension is
surjective if and only if it is dominant if and only if it is generically
finite if and only if there is some point isolated in its fiber.
The first three equivalences are immediate, the last follows from
~\autoref{lemma:isolated-fiber-implies-dominant}.
Since $\dim \prod_{i=1}^m Gr(\codim X - \lambda_i +1, n+1) = \dim \Phi$,
by ~\autoref{lemma:interpolation-dimension},
we have that \ref{interpolation-definition}, \ref{interpolation-dominant},
\ref{interpolation-finite}, \ref{interpolation-isolated} are equivalent.

Next, since $\dim \Phi = \dim \Psi$, and $\Psi$ is irreducible, 
we have
$\dim \prod_{i=1}^m Gr(\codim X - \lambda_i +1, n+1) = \dim \Psi$.
So, by reasoning analogous to that of the previous paragraph, we obtain that \ref{interpolation-pointed}, \ref{interpolation-pointed-dominant}, \ref{interpolation-pointed-finite}, and \ref{interpolation-pointed-isolated} are equivalent.

Next, \ref{interpolation-definition} is equivalent to
\ref{interpolation-naive} because surjectivity of a proper map
of varieties is equivalent to surjectivity on closed points of the varieties.
Since the fibers of the map $\pi_2$ precisely consist of those
elements of $\hilb X$ meeting a specified collection of $q$ points
and a plane $\Lambda$, being surjective is equivalent to there
being some element of $\hilb X$ passing through these $q$ points and meeting $\Lambda$.

Finally, \ref{interpolation-naive} is equivalent to \ref{interpolation-sweep}
because the condition that the variety swept out by the elements of $\hilb X$
containing $q$ points
meet a general plane $\Lambda$ of dimension $\codim X - r$ is equivalent
to the variety swept out by the elements of $\hilb X$ being
$\dim X + r$ dimensional.
This is using the fact that a variety of dimension $d$ in $\bp^n$
meets a general plane of dimension $d'$ if and only if $d + d' \geq n$.
But, of course, the dimension swept out by the elements of $\hilb X$
containing $q$ general points is at most $\dim X + r$ dimensional, because
there is at most an $r$ dimensional space of varieties in $\hilb X$
containing $r$ general points.
This shows the equivalence of properties \ref{interpolation-definition}
through \ref{interpolation-sweep}.

\sssec{Equivalence of Conditions \ref{strong-definition}-\ref{strong-naive-equality}
}

Since $\eta_2$ factors through $\Phi$, and the restriction map
$\Psi \rightarrow \Phi$ is surjective, for all $\lambda$ with $\sum_{i=1}^m \lambda_i = \dim \hilb X$, $\lambda$-interpolation is equivalent
to $\lambda$-pointed interpolation. This establishes the equivalence
of \ref{strong-equality} and \ref{strong-pointed-equality}
and the equivalence of \ref{strong-definition} and \ref{strong-pointed}.

Next, \ref{strong-definition} is equivalent to \ref{strong-naive},
because the map $\pi_2$ contains a point corresponding to a
collection of planes $\Lambda_1, \ldots, \Lambda_m$ in its image
if and only if there is some element of the Hilbert
schemes meeting those planes. Similarly,
\ref{strong-equality} is equivalent to \ref{strong-naive-equality}.

To complete these equivalences, we only need show \ref{strong-naive} is equivalent to \ref{strong-naive-equality}. Clearly \ref{strong-naive} implies \ref{strong-naive-equality}. For the reverse implication, observe that if we start with a collection
of planes $\Lambda_1, \ldots, \Lambda_s$ with $\Lambda_i \in Gr(\codim X - \lambda_i  + 1, n+1)$, so that $\sum_{i=1}^s \lambda_i < \dim \hilb X$, we can extend the sequence $\lambda$ to a sequence $\mu = \left( \mu_1, \ldots, \mu_m \right)$
for $m > s$,
with
$0 \leq \mu_i \leq \codim X$, 
$\mu_i = \lambda_i$ for $i \leq s$, and $\sum_{i=1}^m \mu_i = \dim \hilb X$. Then, if some element of $\hilb X$ meets planes $\Lambda_1, \ldots, \Lambda_m$ corresponding to the sequence $\mu$, it certainly
also meets $\Lambda_1, \ldots, \Lambda_s$. Hence, \ref{strong-naive-equality} implies \ref{strong-naive}.

\sssec{Equivalence of Conditions \ref{cohomological-definition}-\ref{cohomological-boundary}}

The equivalence of \ref{cohomological-definition}
and \ref{cohomological-restatement} is immediate from 
the definitions.
The equivalence of \ref{cohomological-sections}
and \ref{cohomological-vanish}
is a generalization of \cite[Proposition 4.5]{atanasovLY:interpolation-for-normal-bundles-of-general-curves}
to higher dimensional varieties, and the
equivalence of \ref{cohomological-vanish} and
\ref{cohomological-boundary}
is a generalization of \cite[Proposition 4.6]{atanasovLY:interpolation-for-normal-bundles-of-general-curves}
to higher dimensional varieties.
The equivalence of \ref{cohomological-definition}
and \ref{cohomological-strong} is an immediate
generalization of \cite[Theorem 8.1]{atanasov:interpolation-and-vector-bundles-on-curves} to higher dimensional varieties.
To complete the proof, we only need check the equivalence of
Definition \ref{cohomological-definition}
and \ref{cohomological-vanish}.
The forward implication follows immediately from a
couple standard applications of exact sequences,
so we concentrate on the reverse implication.
This essentially follows from a generalization of
\cite[Proposition 4.23]{atanasovLY:interpolation-for-normal-bundles-of-general-curves}, with one minor issue:
We need to check that if we start with a sequence of
	sheaves
	\begin{equation}
		\nonumber
		\begin{tikzcd}
		  0 \ar{r} & F \ar{r} & E \ar{r} & A \ar{r} & 0 
		\end{tikzcd}\end{equation}
	where $A$ has zero dimensional support, then for a general
	collection of points $p_1, \ldots, p_d$ the twisted
	sequence
	\begin{equation}
		\nonumber
		\begin{tikzcd}
		  0 \ar{r} & F \otimes \sci_{p_1, \ldots, p_d} \ar{r} & E \otimes \sci_{p_1, \ldots, p_d}\ar{r} & A \otimes \sci_{p_1, \ldots, p_d}\ar{r} & 0 
		\end{tikzcd}\end{equation}
	remains exact. This held automatically in the case of
\cite[Proposition 4.23]{atanasovLY:interpolation-for-normal-bundles-of-general-curves}, because they were only dealing with the case that the
points were divisors, and hence the ideal sheaves were locally free.
However, here, the resulting sequence is still exact, since
$\stor^1(\sci_{p_1, \ldots, p_d}, A) = 0$,
so long as the points $p_1, \ldots, p_d$ are chosen
to be disjoint from the support of $A$.
We apply this generalization of \cite[Proposition 4.23]{atanasovLY:interpolation-for-normal-bundles-of-general-curves},
and, in that statement,
take $B := Gr(r, \rk N_{X/\bp^n}), E := N_{X/\bp^n}, F := N_{X/\bp^n} \otimes \sci_p, G_b := N_{X/\bp^n}|_p/V_b$,
where 
$V_b$ is the subspace for the corresponding element of $b \in B$.
We then see that all twists of $G$ by the ideal sheaf of a general
set of points either have vanishing $0$ or $1$st cohomology,
implying that $N_{X/\bp^n}$ satisfies interpolation, as in
\ref{cohomological-definition}.

\sssec{Implications Among all Conditions}

By definition \ref{strong-definition} implies \ref{interpolation-definition}.

To complete the proof, we only need to show \ref{cohomological-definition}
implies \ref{strong-pointed} and \ref{interpolation-pointed} (in all characteristics)
and that the reverse implications hold true in characteristic $0$.

For this, choose $\lambda$ with $\sum_{i=1}^m \lambda_i = \dim \hilb X$. We will show that $\lambda$-interpolation of $N_{X/\bp^n}$ implies $\lambda$-pointed
interpolation in all characteristics, and the reverse
implication holds in characteristic $0$.
It suffices to prove this, as this will yield the desired implications.
For example, this implies the relation between \ref{interpolation-pointed} and
\ref{cohomological-definition}, by taking $\lambda = \left( (\codim X)^q, r \right)$.

To see this statement about $\lambda$-pointed interpolation and $\lambda$-interpolation of $N_{X/\bp^n}$,
let $\overline p := (Y, \Lambda_1, \ldots, \Lambda_m, p_1, \ldots, p_m), V_i, \tau$ be as in
~\autoref{proposition:tangent-space-to-psi}.

By \autoref{proposition:tangent-space-to-psi}, we have that the map
$d\eta_2|_{\overline{p}}$ is surjective if and only if
the corresponding map $\tau$
is surjective. But this latter map is precisely that from ~\eqref{equation:sequence-cohomological-interpolation-definition}
in the definition of interpolation for vector bundles,
taking $E := N_{X/\bp^n}$.

So, to complete the proof, it suffices to show that if $d\eta_2|_{\overline{p}}$
is surjective, then $\eta_2$ is surjective, and the converse holds in characteristic $0$.

But now we have reduced this to a general statement about varieties.
Note that $\eta_2$ is a map between two varieties of the same dimension, 
by \autoref{lemma:interpolation-dimension} and that $\overline p$
is a smooth point of $\Psi$ by assumption.
So, it suffices to show that a map between two proper varieties of the same
dimension is surjective if it is surjective on tangent spaces,
and that the converse holds in characteristic $0$.
For the forward implication,
if the map is surjective on tangent spaces,
the map is smooth of
relative dimension $0$ at $\overline p$.
But, this means that $\overline p$ is isolated in its fiber, and so by \autoref{lemma:isolated-fiber-implies-dominant},
we obtain that $\eta_2$ is surjective.

To complete the proof, we only need to show that if $\eta_2$ is surjective and $k$ has characteristic
$0$, then there is a point at which $d\eta_2|_{\overline{p}}$ is surjective.
That is, we only need to show there is a point at which $\eta_2$ is smooth.
But, this follows by generic smoothness, which crucially uses the characteristic $0$ hypothesis!
\end{proof}

\subsection{Complete intersections}

\begin{definition}
	\label{definition:}
	Define $\ci k d n$ to be the closure in the 
	Hilbert scheme of the locus of complete intersections
	of $k$ polynomials of degree $d$ in $\bp^n$.
\end{definition}

\begin{lemma}
	\label{lemma:balanced-complete-intersection}
	Let $k, d, n$ be positive integers.
	Then, $\ci k d n$ satisfies interpolation.
	In particular, any Hilbert scheme of hypersurfaces
	$\ci 1 d n$
	satisfies interpolation.
	Furthermore, interpolation
	is equivalent to meeting $\binom{d+n}{d}-k$
	general points in $\bp^n$.
\end{lemma}
\begin{proof}
	First, observe that $\dim \ci k d n = k(\binom{d+ n}{d} - k)$
	because a point of $\ci k d n$ corresponds to the
	variety cut out by the intersection of all degree $d$ polynomials
	in a k dimensional subspace of $H^0(\bp^n, \sco_{\bp^n}(d))$.
	In other words, there is a birational map between
	the locus of complete intersections and $G(k, H^0(\bp^n, \sco_{\bp^n}(d)))$,
	which is $k(\binom{d+n}{d}-k)$ dimensional.
	So, to show $\ci k d n$ satisfies interpolation, it suffices
	to show there exists such a complete intersection through
	$\binom{d + n}{d} - k$ general points.
	First, since points impose independent conditions on
	degree $d$ hypersurfaces in $\bp^n$, there will indeed be a $k$
	dimensional subspace of $H^0(\bp^n, \sco_{\bp^n}(d))$
	passing through the any collection
	of $\binom{d+n}{d} - k$ points.

	It remains to verify that if the points are chosen
	generally, then the intersection of degree
	$d$ hypersurfaces in the subspace passing through
	the points is a complete intersection.
	To see this, note that the map $\pi_2$
	from \autoref{definition:interpolation}
	is a generically finite map between varieties of
	the same dimension. In particular,
	the element of	
	$G(k, H^0(\bp^n, \sco_{\bp^n}(d)))$
	through a general collection of
	$\binom{d+n}{d}-k$ points
	will be general in
	$G(k, H^0(\bp^n, \sco_{\bp^n}(d)))$.
	Then, since a general element of 
	$G(k, H^0(\bp^n, \sco_{\bp^n}(d)))$
	corresponds to a complete intersection,
	there will indeed be a complete intersection
	passing through a general collection of
	$\binom{d+n}{d}-k$ points.
\end{proof}

\bibliography{master}

\def\cprime{$'$}
\begin{thebibliography}{CFM13}

\bibitem[AH81]{arbarelloH:canonical-curves-and-quadrics-of-rank-4}
Enrico Arbarello and Joseph Harris.
\newblock Canonical curves and quadrics of rank {$4$}.
\newblock {\em Compositio Math.}, 43(2):145--179, 1981.

\bibitem[ALY19]{atanasovLY:interpolation-for-normal-bundles-of-general-curves}
Atanas Atanasov, Eric Larson, and David Yang.
\newblock Interpolation for normal bundles of general curves.
\newblock {\em Mem. Amer. Math. Soc.}, 257(1234):v+105, 2019.

\bibitem[Ata14]{atanasov:interpolation-and-vector-bundles-on-curves}
Atanas Atanasov.
\newblock Interpolation and vector bundles on curves.
\newblock {\em arXiv preprint arXiv:1404.4892v2}, 2014.

\bibitem[Bal17]{ballico2014interpolation}
E.~Ballico.
\newblock An interpolation problem for the normal bundle of curves of genus
  {$g\geq 2$} and high degree in {$\Bbb P^r$}.
\newblock {\em Comm. Algebra}, 45(2):822--827, 2017.

\bibitem[CFM13]{chenFM:effective-divisors-on-moduli-spaces-of-curves-and-abelian-varieties}
Dawei Chen, Gavril Farkas, and Ian Morrison.
\newblock Effective divisors on moduli spaces of curves and abelian varieties.
\newblock 18:131--169, 2013.

\bibitem[Cob22]{coble:associated-sets-of-points}
Arthur~B. Coble.
\newblock Associated sets of points.
\newblock {\em Trans. Amer. Math. Soc.}, 24(1):1--20, 1922.

\bibitem[Cos06a]{coskun:degenerations-of-surface-scrolls}
Izzet Coskun.
\newblock Degenerations of surface scrolls and the {G}romov-{W}itten invariants
  of {G}rassmannians.
\newblock {\em J. Algebraic Geom.}, 15(2):223--284, 2006.

\bibitem[Cos06b]{coskun:the-enumerative-geometry-of-del-Pezzo-surfaces-via-degenerations}
Izzet Coskun.
\newblock The enumerative geometry of {D}el {P}ezzo surfaces via degenerations.
\newblock {\em Amer. J. Math.}, 128(3):751--786, 2006.

\bibitem[Dol04]{dolgachev:on-certain-families-of-elliptic-curves-in-projective-space}
Igor~V. Dolgachev.
\newblock On certain families of elliptic curves in projective space.
\newblock {\em Ann. Mat. Pura Appl. (4)}, 183(3):317--331, 2004.

\bibitem[EH87]{eisenbudH:on-varieties-of-minimal-degree}
David Eisenbud and Joe Harris.
\newblock On varieties of minimal degree (a centennial account).
\newblock In {\em Algebraic geometry, {B}owdoin, 1985 ({B}runswick, {M}aine,
  1985)}, volume~46 of {\em Proc. Sympos. Pure Math.}, pages 3--13. Amer. Math.
  Soc., Providence, RI, 1987.

\bibitem[EL92]{einL:stability-and-restrictions-of-picard-bundles}
Lawrence Ein and Robert Lazarsfeld.
\newblock Stability and restrictions of {P}icard bundles, with an application
  to the normal bundles of elliptic curves.
\newblock In {\em Complex projective geometry ({T}rieste, 1989/{B}ergen,
  1989)}, volume 179 of {\em London Math. Soc. Lecture Note Ser.}, pages
  149--156. Cambridge Univ. Press, Cambridge, 1992.

\bibitem[ELB89]{elencwajg:l-anneau-de-chow-des-triangles-du-plan}
G.~Elencwajg and P.~Le~Barz.
\newblock L' anneau de chow des triangles du plan.
\newblock {\em Compositio Mathematica}, 71(1):85--119, 1989.

\bibitem[EP00]{eisenbudP:the-projective-geometry-of-the-gale-transform}
David Eisenbud and Sorin Popescu.
\newblock The projective geometry of the {G}ale transform.
\newblock {\em J. Algebra}, 230(1):127--173, 2000.

\bibitem[Gro65]{EGAIV.2}
A.~Grothendieck.
\newblock \'{E}l\'ements de g\'eom\'etrie alg\'ebrique. {IV}. \'{E}tude locale
  des sch\'emas et des morphismes de sch\'emas. {II}.
\newblock {\em Inst. Hautes \'Etudes Sci. Publ. Math.}, (24):231, 1965.

\bibitem[Gro66]{EGAIV.3}
A.~Grothendieck.
\newblock \'{E}l\'ements de g\'eom\'etrie alg\'ebrique. {IV}. \'{E}tude locale
  des sch\'emas et des morphismes de sch\'emas. {III}.
\newblock {\em Inst. Hautes \'Etudes Sci. Publ. Math.}, page 255, 1966.

\bibitem[Har10]{Hartshorne:deformation}
Robin Hartshorne.
\newblock {\em Deformation theory}, volume 257 of {\em Graduate Texts in
  Mathematics}.
\newblock Springer, New York, 2010.

\bibitem[hl]{MO:why-is-the-number-of-irreducible-components-upper-semicontinuous-in-nice-situations}
Aaron~Landesman (http://mathoverflow.net/users/75970/aaron landesman).
\newblock Why is the number of irreducible components upper semicontinuous in
  nice situations?
\newblock MathOverflow.
\newblock URL:http://mathoverflow.net/q/225534 (version: 2015-12-08).

\bibitem[HP95]{harrisP:severi-degrees-in-cogenus-3}
Joseph Harris and Rahul Pandharipande.
\newblock Severi degrees in cogenus 3.
\newblock {\em arXiv preprint arXiv:alg-geom/9504003v1}, 1995.

\bibitem[IP99]{iskovskikhP:fano-varieties}
V.~A. Iskovskikh and Yu.~G. Prokhorov.
\newblock Fano varieties.
\newblock In {\em Algebraic geometry, {V}}, volume~47 of {\em Encyclopaedia
  Math. Sci.}, pages 1--247. Springer, Berlin, 1999.

\bibitem[Lan18]{landesman:interpolation-of-varieties-of-minimal-degree}
Aaron Landesman.
\newblock Interpolation of varieties of minimal degree.
\newblock {\em Int. Math. Res. Not. IMRN}, (13):4063--4083, 2018.

\bibitem[Lar16]{larson:interpolation-for-restricted-tangent-bundles-of-general-curves}
Eric Larson.
\newblock Interpolation for restricted tangent bundles of general curves.
\newblock {\em Algebra Number Theory}, 10(4):931--938, 2016.

\bibitem[LP16]{landesman:undergraduate-thesis}
Aaron Landesman and Anand Patel.
\newblock Interpolation in algebraic geometry.
\newblock {\em arXiv preprint arXiv:1605.01117v1}, 2016.

\bibitem[PP15]{pereiraP:an-invitation-to-web-geometry}
Jorge~Vit{\'o}rio Pereira and Luc Pirio.
\newblock {\em An invitation to web geometry}, volume~2 of {\em IMPA
  Monographs}.
\newblock Springer, Cham, 2015.

\bibitem[Ran07]{ran:normal-bundles-of-rational-curves-in-projective-spaces}
Ziv Ran.
\newblock Normal bundles of rational curves in projective spaces.
\newblock {\em Asian J. Math.}, 11(4):567--608, 2007.

\bibitem[Rus03]{russell:counting-singular-plane-curves-via-hilbert-schemes}
Heather Russell.
\newblock Counting singular plane curves via {H}ilbert schemes.
\newblock {\em Adv. Math.}, 179(1):38--58, 2003.

\bibitem[Rus04]{russell:degenerations-of-monomial-ideals}
Heather Russell.
\newblock Degenerations of monomial ideals.
\newblock {\em Math. Res. Lett.}, 11(2-3):231--249, 2004.

\bibitem[Sac80]{sacchiero:normal-bundles-of-rational-curves-in-projective-space}
Gianni Sacchiero.
\newblock Normal bundles of rational curves in projective space.
\newblock {\em Ann. Univ. Ferrara Sez. VII (N.S.)}, 26:33--40 (1981), 1980.

\bibitem[Ser06]{sernesi:deformations-of-algebraic-schemes}
Edoardo Sernesi.
\newblock {\em Deformations of algebraic schemes}, volume 334 of {\em
  Grundlehren der Mathematischen Wissenschaften [Fundamental Principles of
  Mathematical Sciences]}.
\newblock Springer-Verlag, Berlin, 2006.

\bibitem[Sha13]{shafarevich:basic-algebraic-geometry-1}
Igor~R. Shafarevich.
\newblock {\em Basic algebraic geometry. 1}.
\newblock Springer, Heidelberg, third edition, 2013.
\newblock Varieties in projective space.

\bibitem[Ste89]{stevens:on-the-number-of-points-determining-a-canonical-curve}
Jan Stevens.
\newblock On the number of points determining a canonical curve.
\newblock {\em Nederl. Akad. Wetensch. Indag. Math.}, 51(4):485--494, 1989.

\bibitem[Ste96]{stevens:on-the-computation-of-versal-deformations}
J.~Stevens.
\newblock On the computation of versal deformations.
\newblock {\em J. Math. Sci.}, 82(5):3713--3720, 1996.
\newblock Topology, 3.

\bibitem[Ste03]{stevens:deformations-of-singularities}
Jan Stevens.
\newblock {\em Deformations of singularities}, volume 1811 of {\em Lecture
  Notes in Mathematics}.
\newblock Springer-Verlag, Berlin, 2003.

\bibitem[Vak]{vakil:foundations-of-algebraic-geometry}
Ravi Vakil.
\newblock {\itshape MATH 216: Foundations of Algebraic Geometry}.

\end{thebibliography}

\bibliographystyle{alpha}

\end{document}